\documentclass[11pt,reqno]{amsart}
\usepackage{amsmath,amssymb,mathrsfs,amscd}
\usepackage{graphicx,subfigure,cite}
\usepackage{bm}
\usepackage{color}
\usepackage[mathscr]{euscript}

\newtheorem{theorem}{Theorem}[section]
\newtheorem{corollary}[theorem]{Corollary}
\newtheorem{lemma}[theorem]{Lemma}

\newtheorem{example}{Example}

\numberwithin{figure}{subsection}
\numberwithin{equation}{section}

\allowdisplaybreaks[4]

\begin{document}

\title[An inverse source problem]{An inverse source problem for the stochastic multi-term time-fractional diffusion-wave equation}

\author{Xiaoli Feng}
\address{School of Mathematics and Statistics, Xidian University, Xi'an, 713200, China}
\email{xiaolifeng@xidian.edu.cn}

\author{Qiang Yao}
\address{School of Mathematics and Statistics, Xidian University, Xi'an, 713200, China}
\email{yqiang@stu.xidian.edu.cn}

\author{Peijun Li}
\address{Department of Mathematics, Purdue University, West Lafayette, Indiana 47907, USA}
\email{lipeijun@math.purdue.edu}

\author{Xu Wang}
\address{LSEC, ICMSEC, Academy of Mathematics and Systems Science, Chinese Academy of Sciences, Beijing 100190, China, and School of Mathematical Sciences, University of Chinese Academy of Sciences, Beijing 100049, China}
\email{wangxu@lsec.cc.ac.cn}

\thanks{XF and QY are supported by the Natural Science Basic Research Program of Shaanxi (No. 2023-JC-YB-054). PL is supported in part by the NSF grant DMS-2208256. XW is supported by NNSF of China (11971470, 11871068, and 12288201) and by CAS Project for Young Scientists in Basic Research (YSBR-087).}



\begin{abstract}
In this paper, we study both the direct and inverse random source problems associated with the multi-term time-fractional diffusion-wave equation driven by a fractional Brownian motion. Regarding the direct problem, the well-posedness is established and the regularity of the solution is characterized for the equation. In the context of the inverse problem, the uniqueness and instability are investigated on the determination of the random source. Furthermore, a reconstruction formula is provided for the phaseless Fourier modes of the diffusion coefficient in the random source, based on the variance of the boundary data. To reconstruct the time-dependent source function from its phaseless Fourier modes, the PhaseLift method, combined with a spectral cut-off regularization technique, is employed to tackle the phase retrieval problem. The effectiveness of the proposed method is demonstrated through a series of numerical experiments.
\end{abstract}

\maketitle

\section{Introduction}

Time-fractional differential equations (TFDEs) have a wide range of applications across diverse fields, including mathematics, physics, engineering, biology, and finance. They offer valuable tools for modeling complex phenomena characterized by memory effects, non-local behaviors, and anomalous diffusion processes. The extensive body of research on TFDEs reflects their importance. For instance, Schneider et al. \cite{schneider1989fractional} conducted an analysis of the Green function of TFDEs with a fractional derivative order of $\alpha$ ($0<\alpha<1$). This Green function exhibits a mean-squared displacement resembling $Ct^\alpha$, a property considered essential for capturing sub-diffusion phenomena. In a separate study, Giona et al. \cite{giona1992fractional} highlighted the effectiveness of TFDEs in describing relaxation phenomena within complex viscoelastic materials. A multitude of other instances of mathematical, numerical studies, and applications of TFDEs can be found in references such as \cite{metzler2000random, jahanshahi2021fractional, singh2021analysis, sakamoto2011initial} and the associated citations.

In certain practical applications of TFDEs, it has been observed that the order of fractional derivatives in some models can vary within the range of $(0, 2)$, as demonstrated in studies such as \cite{coimbra2003mechanics}. Additionally, sub-diffusion processes exhibiting a logarithmic growth of mean-squared displacement, such as $C\log(t)$, have been introduced in research examples, as evidenced by \cite{chechkin2002retarding}. To address these scenarios, researchers have introduced the concept of distributed order TFDEs. In these equations, the distributed order derivative is obtained from a specified interval, such as $[0, 1]$ or even $[0, 2]$, and is weighted using a positive weighting function. A special case of the distributed order TFDE is the multi-term TFDE. Multi-term TFDEs provide a versatile framework to model systems exhibiting multiple relaxation or memory time scales, which are common in various natural phenomena.

In this paper, we consider the following stochastic multi-term time-fractional diffusion-wave equation driven by a fractional Brownian motion (fBm):
\begin{equation}\label{mainfunction}
\left\{
\begin{aligned}
&\sum\limits_{k=1}^{n}\partial_{t}^{\alpha_{k}}u(x,t)-\partial_{xx}u(x,t)=f(t)\dot{B}^H(x), && (x,t)\in{D}\times \mathbb{R}_{+},\\
&u(x,0)=0, && x\in\overline{{D}},\\
&\partial_tu(x,0)=0, && x\in\overline{{D}},\quad \text{if}\ \ \alpha_n\in(1,2], \\
&\partial_{x}u(0,t)=0, \ u(1,t)=0, && t\in\mathbb{R}_{+},
\end{aligned}
\right.
\end{equation}
where $\partial_{t}^{\alpha_{k}}$ denotes the Caputo fractional derivative with order $\alpha_{k}\in{(0,2]}$ satisfying the condition $\alpha_{1}<\alpha_{2}<\cdots<\alpha_{n}$ and $\alpha_1<2$. The domain $D:=(0,1)$. The diffusion coefficient $f$, also referred to as the source function, is a deterministic function with the property $f(0) = 0$. In this context, $B^H$ represents the spatial fBm with a Hurst index $H\in(0,1)$, and $\dot{B}^H$ represents the formal derivative of $B^H$ with respect to the spatial variable $x$. Further details about the fBm and the Caputo fractional derivative will be provided in Section \ref{sec:2}.

Two distinct problems are associated with \eqref{mainfunction}: the direct source problem and the inverse source problem. In the direct problem, given $f(t)$, the objective is to solve the initial-boundary problem \eqref{mainfunction} to determine the solution $u$. On the other hand, the inverse source problem seeks to determine the source function $f$ from the boundary data $\{u(0,t)\}_{t\in\mathbb R+}$. This paper is dedicated to addressing both the direct and inverse source problems. To simplify our analysis without loss of generality, we focus on random sources in the form of $f(t)\dot B^H(x)$ rather than considering a more general setting, which includes a deterministic function $g(x,t)$, given as $g(x,t)+f(t)\dot B^H(x)$. In the case of random sources in this general form, $g(x,t)+f(t)\dot{B}^H(x)$, taking the expectation of both sides of the equation transforms the problem into a deterministic one involving an unknown source $g(x,t)$. Notably, productive results have already been achieved in solving such deterministic inverse source problems. Additionally, we assume that the initial and boundary conditions are homogeneous, as is commonly done, since nonhomogeneous conditions can be converted into homogeneous ones using standard techniques in partial differential equations.

Extensive research has been undertaken on deterministic sources in the context of multi-term TFDEs, including both direct and inverse source problems. For instance, in \cite{daf2008}, multivariate Mittag--Leffler functions were utilized to represent solutions of the initial-boundary value problem for multi-term TFDEs with constant coefficients when $1<\alpha_k<2$. In \cite{jiang2012analytical}, innovative techniques were introduced for deriving analytical solutions for multi-term TFDEs. The investigation of well-posedness and long-term asymptotic behavior of the initial boundary value problem for multi-term TFDEs was addressed in \cite{li2015initial}. A strong maximum principle for multi-term TFDEs was established in \cite{liu2017strong}, demonstrating uniqueness in determining the temporal component of the source term. In \cite{li2019identification}, the authors focused on the identification of time-dependent source terms in multi-term TFDEs using boundary data. In \cite{jiang2020inverse}, successful recovery of spatially dependent sources in multi-term TFDEs was achieved using final data.

Due to the inherent uncertainties present in practical problems, researchers have directed considerable attention to the investigation of stochastic models. In contrast to deterministic inverse problems, stochastic inverse problems are confronted with additional challenges due to ill-posedness and the presence of randomness and uncertainty. Notably, significant progress has been made in the field of inverse random source scattering problems, particularly in the context of stochastic wave equations driven by white noise. In this regard, effective computational models have been developed, as demonstrated by the contributions of \cite{bao2016inverse, bao2017inverse, bao2010numerical, li2018inverse, li2021inverse}. These works primarily center on the reconstruction of statistical characteristics related to random sources, including parameters such as mean and variance.

In recent years, there has been significant progress in the field of inverse random source problems for TFDEs. Research efforts have mainly focused on addressing two distinct categories of noise: time-dependent and spatial-dependent. In the case of time-dependent noise, Niu et al. \cite{niu2020inverse} explored scenarios involving a random source expressed as $f(x)h(t) + g(x)\dot{B}(t)$ by utilizing statistical information derived from the final data. Expanding upon this line of research, Feng et al. \cite{feng2020inverse} broadened the scope of their investigation to contain situations featuring a random source in the form of $f(x)h(t) + g(x)\dot{B}^H(t)$. Additionally, in \cite{liu2020reconstruction, Fu+2021}, the specific case involving a random source $f(x)(g_1(t) + g_2(t)\dot{B}(t))$ was examined. Lassas et al. \cite{LassasLiZhang2023ip} studied the general case characterized by a random source given by $I_t^{\delta}(f_1(x)g_1(t) + f_2(x)g_2(t)\dot{B}(t))$, where $I_t^{\delta}$ represents the Riemann--Liouville fractional integral operator. In contrast, research on spatial-dependent noise in this context is relatively limited. An exception to this is the work of Gong et al. \cite{gong2021numerical}, where they conducted a detailed analysis of a TFDE with $\alpha\in(0,1)$, characterized by a random source represented as $f(t)\dot{B}(x)$.

This paper is dedicated to solving the problem \eqref{mainfunction} involving the multi-term TFDE with the random source $f(t)\dot{B}^H(x)$. When the Hurst parameter $H=\frac12$, the fractional Brownian motion (fBm) ${B}^H(x)$ reduces to the classical Brownian motion ${B}(x),$ and the random source becomes $f(t)\dot{B}(x)$. Furthermore, the multi-term TFDE considered here can also be reduced to the single-term case, i.e., the classical TFDE with the fractional derivative order $0<\alpha<2$, which includes sub-diffusion ($0<\alpha<1$), super-diffusion ($1<\alpha<2$), and heat conduction ($\alpha=1$). In this context, our research can be viewed as an extension of the prior work presented in \cite{gong2021numerical}. The approach of transforming problem \eqref{mainfunction} into a stochastic boundary value problem in the frequency domain, followed by the application of the PhaseLift method to recover the time-dependent source function, is inspired by the literature \cite{gong2021numerical}. The central challenge of our study lies in dealing with the complexities introduced by fBm. To date, there is limited research on the inverse random source problem with a source driven by fBm. In \cite{feng2020inverse}, an inverse random source problem was examined for the TFDE with $0<\alpha<1$, where the random source was given as $f(x)h(t)+g(x)\dot{B}^H(t)$. To handle stochastic integration,   the moving average representation of fBm was employed to convert the variance of random integrals into deterministic integrals. Nevertheless, the integrals transformed through this technique proved to be particularly complex and involved addressing a significant number of singular integrals.

In this study, we adopt the harmonizable representation, as described in \cite{dasgupta1997fractional}, to address our problem. In contrast to the moving average representation of fBm, the harmonizable representation enables us to express the variance of random integrals in a more concise manner. A fundamental prerequisite for this approach is that the integrand function must satisfy specific properties. However, it is worth noting that in our theoretical analysis, the integrand function is closely associated with both the Green function \eqref{eq:g} and the Hurst parameter $H$. The Green function \eqref{eq:g} exhibits a complex form and has somewhat limited desirable properties, presenting a challenge. Additionally, the Hurst parameter, which falls within the range of $0<H<1$, can lead to the emergence of singular integrals in \eqref{eq:rep}, especially when $0<H<\frac12$. Through subsequent analysis, we derive a crucial isometry formula, as presented in Lemma \ref{isometry}. Utilizing this formula, we establish the well-posedness of the solution for the direct problem, and concurrently, we prove the uniqueness and characterize the ill-posed nature of the inverse problem.

To validate our theoretical findings, we conduct numerical experiments for \eqref{mainfunction} with two time fractional terms, addressing cases of both sub-diffusion and super-diffusion simultaneously. We employ a finite difference scheme to solve the direct problem and obtain the boundary data. For the inverse problem, given that the available data is the modulus in the frequency domain and the problem is inherently ill-posed, it necessitates the resolution of a phase retrieval problem. To address this challenge, we adopt the PhaseLift method, in conjunction with a spectral cut-off technique, to reconstruct the source function. Our numerical results demonstrate the effectiveness of the proposed method in handling both smooth and nonsmooth source functions.

The remaining sections of this paper are structured as follows. Section \ref{sec:2} provides some necessary background information to facilitate the main results. The well-posedness of the direct problem is demonstrated in Section \ref{sec:3}. Section \ref{sec:4} provides a proof of uniqueness and a characterization of the ill-posed nature of the inverse problem. In Section \ref{sec:5}, we introduce numerical methods for solving the direct problem, along with detailed information regarding the PhaseLift method for solving the inverse problem, supported by numerical examples to confirm the theoretical results. Finally, in Section \ref{sec:6}, we conclude with a summary of our study and offer suggestions for future research directions.

\section{Preliminaries}\label{sec:2}

In this section, we provide a brief introduction to fractional Brownian motion and the Caputo fractional derivative, both of which are employed in the context of this study.

\subsection{Fractional Brownian motion}

A centered Gaussian process $B^H = \{B^H(x): x\in\mathbb R\}$, defined on a probability space that comprises a complete triple $(\Omega,\mathcal F, \mathbb P)$, is referred to as a fractional Brownian motion (fBm) characterized by a Hurst index $H\in(0, 1)$ when it exhibits the covariance function, as presented in \cite[Chapter 5.1]{N06} or \cite[Definition 7.2.2]{ST94}:
\[
\mathcal R_H(x,y):=\mathbb{E}\left[B^H(x)B^H(y)\right]=\frac{1}{2}\left(|x|^{2H}+|y|^{2H}-|x-y|^{2H}\right),\quad x,y\in\mathbb{R}.
\]
It can be readily verified that $B^H$ exhibits self-similarity with an index of $H$, expressed as $B^H(ax)\overset{d}{=}a^{H}B^H(x)$ for any $a>0$. Additionally, it possesses stationary increments, as indicated by $B^H(x+h)-B^H(x)\overset{d}{=}B^H(h)-B^H(0)$ for any $x,h\in\mathbb{R}$. Here, the notation $X\overset{d}{=}Y$ denotes that random variables $X$ and $Y$ share the same probability distribution.

In order to streamline the formulation of moments for the stochastic integral, we introduce a specific integral representation of $B^H$ in relation to a complex Gaussian measure defined over the entire real line $\mathbb{R}$.
	
\begin{lemma}[cf. \cite{ST94}]
The fBm $B^H$ with $H\in(0,1)$ has the integral representation given by
\begin{equation}\label{eq:rep}
B^H(x)=C_H \int_{-\infty}^{\infty} \frac{e^{{\rm i} \lambda x}-1}{{\rm i} \lambda|\lambda|^{H-\frac{1}{2}}} \mathrm{d} \widetilde{W}(\lambda),\quad x \in \mathbb{R},
\end{equation}
where the constant $C_H$ is defined as
\begin{equation*}
C_H:=\left(\frac{H\Gamma(2H)\sin(H\pi)}{\pi}\right)^{\frac{1}{2}},
\end{equation*}
and $\widetilde{W}:=W_1+\mathrm{i}W_2$ represents a complex Gaussian measure. Here, $W_1$ and $W_2$ are independent Gaussian measures that are independently scattered over $\mathbb{R}_+$, and they satisfy the properties $W_1(A)=W_1(-A)$ and $W_2(A)=-W_2(-A)$ for any Borel set $A$ of finite Lebesgue measure.
\end{lemma}

The expression presented in \eqref{eq:rep} is commonly referred to as the harmonizable representation, also known as the spectral representation (cf. \cite{N06}). This representation is derived from the moving average representation of fBms using the Parseval identity. We refer to \cite{ST94} and \cite{N06} for further elaboration on these representations over the real line $\mathbb{R}$ and over finite intervals, respectively.

\subsection{Caputo fractional derivative}

The Caputo fractional derivative, which is one of the methods for computing fractional derivatives, was introduced by M. Caputo in 1967. To begin, let us revisit the definition of the $\nu$th order Caputo fractional derivative of a function $v$, denoted as
\begin{equation*}
_0D_t^\nu v(t)=\frac{1}{\Gamma(n-\nu)}\int_0^t(t-\xi)^{n-\nu-1}v^{(n)}(\xi)\mathrm{d}\xi,
\end{equation*}
where the Gamma function $\Gamma(\alpha)=\int_0^{\infty}e^{-s}s^{\alpha_1}\mathrm{d}s$, and $n=\lceil \nu \rceil$ with $\lceil \cdot \rceil$ denoting the smallest positive integer that is larger than or equal to $\nu$.

Next, we examine the Fourier transform of the Caputo fractional derivative.

\begin{lemma}[cf.\cite{das2011functional}]\label{lm:fracd}
Let $\nu\in\mathbb{R_+}$, and consider a causal function $v(t)$, where $v(t) = 0$ for $t\le0$, and $\sum_{k=0}^{n-1}|v^{(k)}(0)|=0$ with $n=\lceil \nu \rceil$. Additionally, assume that $v^{(k)}(t)$ has compact support for all $k=0,\ldots,n$. Under these conditions, the fractional derivative $_0D_t^\nu v(t)$ of $v(t)$ is well-defined in $L^2(\mathbb{R})$, and its Fourier transform satisfies
\begin{equation*}
\mathscr{F}[_0D_t^\nu v(\cdot)](\omega) = (\mathrm{i}\omega)^\nu\mathscr{F}[v(\cdot)](\omega),
\end{equation*}
where the Fourier transform of $v$ is denoted as
\begin{equation*}
\hat{v}(\omega):=\mathscr{F}[v(\cdot)](\omega)=\int_\mathbb{R}v(t)\mathrm{e}^{-\mathrm{i}\omega t}\mathrm{d}t.
\end{equation*}
\end{lemma}

\section{The direct problem}\label{sec:3}

The goal of this section is to establish the well-posedness of the direct source problem. To achieve this, we convert it into an equivalent problem in the frequency domain. We then proceed to demonstrate estimates for the corresponding Green function, including an analogue of It{\^o} isometry for its stochastic integral with respect to the fBm. Subsequently, we investigate the existence, uniqueness, and regularity of the mild solution for this equivalent problem. Based on these findings, we attain the well-posedness of the time-domain problem \eqref{mainfunction}.

\subsection{The problem in the frequency domain}\label{s3.1}

For a function $v\in L^2(\mathbb R_+)$, we consider its zero extension outside of $\mathbb R_+$, still denoted by $v$, such that its Fourier transform $\hat v$ is well-defined. Taking the Fourier transform of \eqref{mainfunction} and applying Lemma \ref{lm:fracd}, we obtain the following stochastic differential equation in the frequency domain, where $x\in D$ and $\omega\in\mathbb R$:
\begin{equation}\label{freqmaineq}
\left\{
\begin{array}{ll}
\hat{u}_{xx}(x,\omega)-\sum\limits_{k=1}^{n}(\mathrm{i}\omega)^{\alpha_{k}}\hat{u}(x,\omega)=-\hat{f}(\omega)\dot{B}^H(x), \\
\partial_{x}\hat{u}(0,\omega)=0,\quad \hat{u}(1,\omega)=0,
\end{array}
\right.
\end{equation}
where $\hat u$ and $\hat f$ are the Fourier transforms of $u$ and $f$ with respect to $t$, respectively.

Note that the complex number $(\mathrm{i}\omega)^{\alpha_{k}}$ may be multi-valued when $\alpha_k$ is a fractional number. Throughout this paper, we always adopt its principal value and represent it as
\begin{equation}\label{sdefinition}
s:=\sum\limits_{k=1}^{n}(\mathrm{i}\omega)^{\alpha_{k}}            =\sum\limits_{k=1}^{n}|\omega|^{\alpha_{k}}\mathrm{e}^{\mathrm{i}\frac{\pi\alpha_k}{2}\mathrm{sgn}(\omega)}
=\sum\limits_{k=1}^{n}|\omega|^{\alpha_{k}}\left(\cos\left(\frac{\pi\alpha_k}{2}\right)+\text{sgn}(\omega)\mathrm{i}\sin\left(\frac{\pi\alpha_k}{2}\right)\right),
\end{equation}
where $\mathrm{sgn}(\cdot)$ denotes the sign function. The parameter $s$ has the following properties.

\begin{lemma}\label{lm:s}
For the parameter $s$ as defined in \eqref{sdefinition}, it holds that $s=0$ if and only if $\omega=0$. Furthermore,
\begin{equation*}
|s|\geqslant\sin\left(\frac{\pi\alpha_{\max}}{2}\right)|\omega|^{\alpha_{\max}},
\end{equation*}
where $\alpha_{\max}:=\max_{i=1}^n\{\alpha_i : \alpha_i\neq2\}$.
\end{lemma}

\begin{proof}
If $\omega=0$, it is evident that $s=0$. It is adequate to demonstrate that $\omega=0$ when $s=0$.

Assuming, by contradiction, that $\omega\neq0$, it is important to note that $\frac{\pi\alpha_k}{2}\in(0,\pi]$ and, therefore, $\sin\left(\frac{\pi\alpha_k}{2}\right)\geqslant0$ for any $\alpha_k\in(0,2]$ and $k=1,\cdots,n$. If
\[ s=\sum\limits_{k=1}^{n}|\omega|^{\alpha_{k}}\left(\cos\left(\frac{\pi\alpha_k}{2}\right)+\text{sgn}(\omega)\mathrm{i}\sin\left(\frac{\pi\alpha_k}{2}\right)\right)=0,
\]
then its imaginary part $\Im[s]$ must also be zero, i.e.,
\[ \Im[s]=\text{sgn}(\omega)\mathrm{i}\sum\limits_{k=1}^{n}|\omega|^{\alpha_{k}}\sin\left(\frac{\pi\alpha_k}{2}\right)=0,
\]
which implies that $\alpha_k=2$ for all $k=1,\cdots,n$. Substituting $\alpha_k=2$ into the expression for $s$, we obtain
\[ s=\sum\limits_{k=1}^{n}|\omega|^2\cos\left(\pi\right)=-n\omega^2=0,
\]
which leads to a contradiction to the assumption $\omega\neq0$. We then establish the equivalence between $s=0$ and $\omega=0$.

Next, we proceed to estimate the lower bound of $|s|$. Given that $0<\alpha_1<\alpha_2<\cdots<\alpha_n\leqslant2$, we have
\[
\sin\left(\frac{\pi\alpha_k}{2}\right)>0,\quad k=1,\cdots,n-1,
\]
and $\sin\left(\frac{\pi\alpha_n}{2}\right)\geqslant0$. Consequently, we can deduce
\[
|s|\geqslant\left|\Im{[s]}\right|=\sum\limits_{k=1}^n\left|\omega\right|^{\alpha_{k}}\sin\left(\frac{\pi\alpha_k}{2}\right)\geqslant\sin\left(\frac{\pi\alpha_{\max}}{2}\right)|\omega|^{\alpha_{\max}},
\]
which completes the proof.
\end{proof}

\subsection{The Green function}

In order to establish the well-posedness of \eqref{freqmaineq} and derive an explicit solution, we begin by introducing the Green function and its associated energy estimates.
	
Let $G_\omega(x, y)$ be the Green function of \eqref{freqmaineq} for any fixed $\omega\in\mathbb{R}$. It solves the following problem (cf. \cite{gong2021numerical}):
\begin{equation*}
\begin{cases}\partial_{xx} G_\omega(x,y)-s G_\omega(x,y)=\delta(x-y), \\ 
\partial_x G_\omega(0, y)=0, \quad G_\omega(1, y)=0, 
\end{cases}
\end{equation*}
where $x, y\in D$ and the frequency $s$ is given in \eqref{sdefinition}. It is shown in \cite{gong2021numerical} that the Green function $G_\omega(x,y)$ admits the following expression:
\begin{equation}\label{eq:g}
G_\omega(x,y)=\left\{
\begin{aligned}
&\max\{x,y\}-1,\quad&\omega=0, \\ &\frac{\mathrm{e}^{\sqrt{s}(x+y)}+\mathrm{e}^{\sqrt{s}|x-y|}-\mathrm{e}^{\sqrt{s}(2-x-y)}-\mathrm{e}^{\sqrt{s}(2-|x-y|)}}{2 \sqrt{s}\left(1+\mathrm{e}^{2\sqrt{s}}\right)},\quad&\omega\neq0,
\end{aligned}
\right.
\end{equation}
where we choose the principal value for $\sqrt{s}=|s|^{\frac12}\mathrm{e}^{{\rm i}\frac{\arg(s)}2}$ with $\arg(\cdot)$
representing the argument with a radiant principal value in $(-\pi,\pi]$ to ensure that its real part $\Re[\sqrt s]>0$ for $\omega\in\mathbb{R}\setminus\{0\}$.
	
\begin{lemma}\label{gnorm}
For any $x\in \overline{D}$, the Green function $G_\omega$ provided in \eqref{eq:g} satisfies
\begin{equation*}
\sup_{\omega\in\mathbb{R}}\left\|G_\omega(x,\cdot)\right\|_{L^2(D)}\leqslant C,\quad\sup_{\omega\in\mathbb{R}}\left\|G_\omega\right\|_{L^2(D\times D)}\leqslant C.
\end{equation*}
Additionally, as $|s|\to\infty$, the following inequalities hold:
\[
\|G_\omega(x,\cdot)\|_{L^2(D)}\leqslant C|s|^{-\frac12},\quad\|G_\omega\|_{L^2(D\times D)}\leqslant C|s|^{-\frac12},
\]
where $C$ denotes positive constants that are independent of $\omega$ and $x$.
\end{lemma}

\begin{proof}
It is only necessary to demonstrate that the above results apply to $\|G_\omega(x,\cdot)\|_{L^2(D)}$, as the results for $\|G_\omega\|_{L^2(D\times D)}$ follow directly.

If $\omega=0$, it is evident that for any $x\in \overline{D}$
\[
\|G_0(x,\cdot)\|_{L^2(D)}^2=\int_0^x(x-1)^2{\rm d}y+\int_x^1(y-1)^2{\rm d}y=\frac{(x-1)^2(2x+1)}3\leqslant\frac13.
\]

If $\omega\neq0$, using \eqref{eq:g} and noting the inequality for $x, y\in(0, 1)$ (cf. \cite[$(3.4)$]{gong2021numerical}):
\begin{equation}\label{eq:gw}
|G_\omega(x, y)|^2\leq \frac{\mathrm{e}^{2\Re[\sqrt{s}](x+y)}+\mathrm{e}^{2\Re[\sqrt{s}]|x-y|}+\mathrm{e}^{2\Re[\sqrt{s}](2-x-y)}+\mathrm{e}^{2\Re[\sqrt{s}](2-|x-y|)}}{|\sqrt{s}\left(1+\mathrm{e}^{2\sqrt{s}}\right)|^2},
\end{equation}
where $\Re[\cdot]$ denotes the real part of a complex number, we deduce from a simple calculation that 
\begin{align}\label{eq:h}
\|G_\omega(x,\cdot)\|_{L^2(D)}^2=&\int_0^1\left|G_\omega(x, y)\right|^2{\rm{d}}y\notag\\
\leqslant &~|s|^{-1}\frac1{|1+\mathrm{e}^{2\sqrt{s}}|^2}\frac{\mathrm{e}^{4\Re[\sqrt{s}]}-1}{\Re[\sqrt{s}]}\notag\\	=&~|s|^{-1}\frac{1}{1+\mathrm{e}^{4\Re{[\sqrt{s}]}}+2\mathrm{e}^{2\Re{[\sqrt{s}]}}\cos\left(2\Re{[\sqrt{s}]}\tan\left(\frac{\arg(s)}2\right)\right)}\frac{\mathrm{e}^{4\Re{[\sqrt{s}]}}-1}{\Re[\sqrt{s}]}\notag\\
=&:|s|^{-1}h\left(\Re[\sqrt{s}]\right),
\end{align}
where $h$ is a positive function for any $\Re[\sqrt{s}]>0$. Note that
\begin{equation}\label{eq:hbound}
\begin{aligned}
\left|h\left(\Re[\sqrt{s}]\right)\right| =&\left|\frac{1}{\Re[\sqrt{s}]}\frac{\mathrm{e}^{4\Re{[\sqrt{s}]}}-1}{1+\mathrm{e}^{4\Re{[\sqrt{s}]}}+2\mathrm{e}^{2\Re{[\sqrt{s}]}}\cos\left(2\Re{[\sqrt{s}]}\tan\left(\frac{\arg(s)}2\right)\right)}\right|\\ \leqslant&\left|\frac{1}{\Re[\sqrt{s}]}\frac{\mathrm{e}^{4\Re{[\sqrt{s}]}}-1}{1+\mathrm{e}^{4\Re{[\sqrt{s}]}}-2\mathrm{e}^{2\Re{[\sqrt{s}]}}}\right|\\ =&\left|\frac{1}{\Re[\sqrt{s}]}\frac{\mathrm{e}^{2\Re{[\sqrt{s}]}}+1}{\mathrm{e}^{2\Re{[\sqrt{s}]}}-1}\right|\leqslant\frac{2}{\Re[\sqrt{s}]}\to0,\quad\Re[\sqrt{s}]\to\infty
\end{aligned}
\end{equation}
and
\[		\lim_{\Re[\sqrt{s}]\to0}h\left(\Re[\sqrt{s}]\right)=1.
\]
Following the same procedure as in \cite[Lemma 3.1]{gong2021numerical}, we can obtain the uniform boundedness of the function $h$ over $[0,\infty)$. Hence,
\begin{equation}\label{eq:gbound}
\|G_\omega(x,\cdot)\|_{L^2(D)}\leqslant C|s|^{-\frac12}\quad\forall~\omega\in\mathbb R.
\end{equation}
		
On the other hand, for any fixed $x\neq y$, when considering $G_\omega(x,y)$ as a function of $s$, it is analytic with respect to $s$ and is continuous at $s=\omega=0$, implying that
\[
\lim_{\omega\to0}G_\omega(x,y)=G_0(x,y).
\]
As a consequence,
\[
\|G_\omega(x,\cdot)\|_{L^2(D)}\leqslant C, \quad |s|\ll 1,
\]
which, together with \eqref{eq:gbound}, finishes the proof.
\end{proof}

For the sake of convenience in notation, we introduce a function $\mathscr{T}G_\omega$ defined in $\mathbb R\times\mathbb R$ as follows:
\begin{equation*}
\mathscr{T}G_\omega(x,y):=
\left\{
\begin{aligned}
& \partial_y G_\omega(x,y), \quad&x\in D,~y\in[0,x)\cup(x,1],\\
& 0, \quad&\text{otherwise}.
\end{aligned}
\right.
\end{equation*}
	
\begin{lemma}\label{gynorm}
The function $\mathscr{T}G_\omega\in L^2(\mathbb R\times\mathbb R)$ is uniformly bounded with respect to $\omega\in\mathbb R$, satisfying
\[
\sup_{\omega\in\mathbb R}\|\mathscr T G_\omega\|_{L^2(\mathbb R\times\mathbb R)}\leqslant C.
\]
Moreover, for any fixed $x\in D$, it holds that
\[
\sup_{\omega\in\mathbb R}\|\mathscr T G_\omega(x,\cdot)\|_{L^2(\mathbb R)}\leqslant C.
\]
In the above expressions, the positive constants denoted by $C$ are independent of both $\omega$ and $x$.
\end{lemma}

\begin{proof}
It suffices to consider $x,y\in D$. If $\omega=0$, then
\begin{equation*}
\mathscr{T}G_0(x,y)=\left\{
\begin{aligned}
&0,\quad y\in[0,x],\\
&1,\quad y\in(x,1].
\end{aligned}
\right.
\end{equation*}
A simple calculation yields
\[
\|\mathscr T G_0(x,\cdot)\|_{L^2(\mathbb R)}^2=\int_x^11dy=1-x\leqslant1
\]
and
\[
\|\mathscr{T} G_0\|_{L^{2}(\mathbb R\times\mathbb R)}^2
=\|\mathscr{T} G_0\|_{L^{2}(D\times D)}^2=\int_0^1\int_x^11^2\mathrm{d}y\mathrm{d}x
=\frac{1}{2}.
\]

If $\omega\neq0$, then for $x\in D,~y\in [0,x)\cup(x,1]$, it holds that
\begin{equation*}		\mathscr{T}G_\omega(x,y)=\frac{\mathrm{e}^{\sqrt{s}(x+y)} - \mathrm{sgn}(x-y) \mathrm{e}^{\sqrt{s}|x-y|} + \mathrm{e}^{\sqrt{s}(2-x-y)} - \mathrm{sgn}(x-y) \mathrm{e}^{\sqrt{s}(2-|x-y|)}}{2 \left(1+\mathrm{e}^{2 \sqrt{s}}\right)}.
\end{equation*}
Similarly, we may obtain from \eqref{eq:gw}--\eqref{eq:h} that 
\begin{equation}\label{eq:Tg}
\begin{aligned}
\|\mathscr{T}G_\omega(x,\cdot)\|_{L^{2}(\mathbb R)}^2\leqslant&~\frac{1}{\left|1+\mathrm{e}^{2\sqrt{s}}\right|^2}\frac{\mathrm{e}^{4\Re[\sqrt{s}]}-1}{\Re[\sqrt{s}]}\\ =&~\frac{1}{1+\mathrm{e}^{4\Re{[\sqrt{s}]}}+2\mathrm{e}^{2\Re{[\sqrt{s}]}}\cos\left(2\Re{[\sqrt{s}]}\tan\left(\frac{\arg(s)}2\right)\right)}\frac{\mathrm{e}^{4\Re{[\sqrt{s}]}}-1}{\Re[\sqrt{s}]}\\
=&~h\left(\Re[\sqrt{s}]\right),
\end{aligned}
\end{equation}
and
\[
\|\mathscr{T}G_\omega\|_{L^{2}(\mathbb{R}\times\mathbb{R})}^2=\int_D\|\mathscr T G_\omega(x,\cdot)\|_{L^2(\mathbb R)}^2{\rm d}x\leqslant\int_D\frac{1}{\left|1+\mathrm{e}^{2\sqrt{s}}\right|^2}\frac{\mathrm{e}^{4\Re[\sqrt{s}]}-1}{\Re[\sqrt{s}]}{\rm d}x = h\left(\Re[\sqrt{s}]\right),
\]
where $h$ is defined in \eqref{eq:h} and is uniformly bounded. Hence, there exists a constant $C>0$ independent of $\omega$ and $x$ such that
\begin{equation*}
\left\|\mathscr{T} G_\omega(x,\cdot)\right\|_{L^2(\mathbb R)}^2\leqslant h(\Re[\sqrt{s}])\leqslant C
\end{equation*}
and
\begin{equation*}
\left\|\mathscr{T} G_\omega\right\|_{L^2(\mathbb{R}\times\mathbb R)}^2\leqslant h(\Re[\sqrt{s}])\leqslant C,
\end{equation*}
which completes the proof.
\end{proof}
	
For any fixed $x\in D$, we denote by $\tilde G_\omega(x,\cdot)$ the zero extension of $G_\omega(x,\cdot)$ outside of $D$. Similarly, we denote $\hat{\tilde G}_\omega(x,\cdot)$ as the Fourier transform of $\tilde G_\omega(x,\cdot)$ with respect to the second variable, and apply the same notation to $\widehat{\mathscr T G}_\omega(x,\cdot)$. Recall the definition of the fractional Sobolev space $\mathscr H^\gamma(\mathbb R)$ with $\gamma\in\mathbb R$ (cf. \cite{Hitchhiker2012Eleonora}):
\begin{equation*}
\mathscr{H}^\gamma\left(\mathbb{R}\right):=\left\{u \in L^2\left(\mathbb{R}\right): \int_{\mathbb{R}}\left(1+|\zeta|^2\right)^\gamma|\hat u(\zeta)|^2 \mathrm d \zeta<\infty\right\},
\end{equation*}
which is equipped with the norm
\[
\|u\|_{\mathscr H^\gamma(\mathbb{R})}:=\left(\int_{\mathbb{R}}\left(1+|\zeta|^{2}\right)^\gamma|\hat u(\zeta)|^2 \mathrm d \zeta\right)^{\frac12}.
\]

Based on the above notations, we establish the following It\^o isometry type equality for the stochastic integral of $G_\omega$ with respect to the fBm. This result is derived using a procedure similar to the one employed in \cite[Chapter 2.2]{dasgupta1997fractional}.

\begin{lemma}\label{isometry}
For any fixed $x\in \overline{D}$ and with $H\in(0,1)$, the stochastic integral $\int_D G_\omega(x,y)\mathrm{d}B^H(y)$ is well-defined and satisfies
\begin{equation}\label{fracequlformula}
\mathbb{E}\left|\int_D G_\omega(x, y) \mathrm{d} B^H(y)\right|^2=C_H^2\int_\mathbb{R}\frac{|\hat{\tilde{G}}_\omega(x,\zeta)|^2}{|\zeta|^{2H-1}}\mathrm{d}\zeta,
\end{equation}
where $C_H$ is defined in \eqref{eq:rep}.
\end{lemma}

\begin{proof}
For the case where $H=\frac{1}{2}$ and $C_H^2=\frac{1}{2\pi}$, the result in \eqref{fracequlformula} follows from the classical It\^{o} isometry and Parseval's theorem, which states that
\[
\mathbb E\left|\int_D G_\omega(x, y) \mathrm{d} B^{\frac12}(y)\right|^2=\int_{\mathbb R}|\tilde G_\omega(x,y)|^2{\rm d}y=\frac1{2\pi}\int_{\mathbb R}|\hat{\tilde{G}}_\omega(x,\zeta)|^2{\rm d}\zeta.
\]

For $H\in(\frac{1}{2},1)$, it is shown in \cite[Section 2.2, Case 1]{dasgupta1997fractional} that \eqref{fracequlformula} holds if $\tilde G_\omega(x,\cdot)\in L^1(\mathbb R)\cap L^2(\mathbb R)$ for any $x\in \overline{D}$. Given that $\tilde G_\omega(x,\cdot)$ is supported in $D$ and $L^2(D)\subset L^1(D)$, it is sufficient to show that $G_\omega(x,\cdot)\in L^2(D)$, which has already been demonstrated in Lemma \ref{gnorm}.

For $H\in(0,\frac12)$, we cannot directly apply the conclusion from \cite[Section 2.2, Case 2]{dasgupta1997fractional} because $\tilde G_\omega(x,\cdot)\notin\mathscr H^1(\mathbb R)$. In the following, we demonstrate that \eqref{fracequlformula} still holds even under a weaker regularity condition for $\tilde G_\omega(x,\cdot)$.

First, we assert that $\tilde{G}_\omega(x,\cdot)\in \mathscr{H}^{\frac12-H}(\mathbb{R})$ for any $H\in(0,\frac12)$. In fact, for any fixed $\epsilon>0$, it follows from Plancherel's theorem that
\begin{align*}
\left\|\tilde G_\omega(x,\cdot)\right\|_{\mathscr{H}^{\frac12-H}(\mathbb{R})} =&\int_\mathbb{R}\left(1+|\zeta|^2\right)^{\frac12-H}|\hat{\tilde{G}}_\omega(x,\zeta)|^2\mathrm{d}\zeta\\
=&\int_{(-\epsilon,\epsilon)}\left(1+|\zeta|^2\right)^{\frac12-H}|\hat{\tilde{G}}_\omega(x,\zeta)|^2\mathrm{d}\zeta + \int_{\mathbb{R}\setminus(-\epsilon,\epsilon)}\left(1+|\zeta|^2\right)^{\frac12-H}|\hat{\tilde{G}}_\omega(x,\zeta)|^2\mathrm{d}\zeta\\ \leqslant&\left(1+\epsilon^2\right)^{\frac12-H}\|\hat{\tilde G}_\omega(x,\cdot)\|^2_{L^2(\mathbb{R})}+\int_{\mathbb{R}\setminus(-\epsilon,\epsilon)}\left(\frac1{|\zeta|^2}+1\right)^{\frac12-H}|\zeta|^{1-2H}|\hat{\tilde{G}}_\omega(x,\zeta)|^2\mathrm{d}\zeta\\
\leqslant&~2\pi\left(1+\epsilon^2\right)^{\frac12-H}\|G_\omega(x,\cdot)\|_{L^2(D)}^2\\
& +\left(\frac1{\epsilon^2}+1\right)^{\frac12-H}\int_{\mathbb{R}\setminus(-\epsilon,\epsilon)}|\zeta|^{-2H-1}|\zeta\hat{\tilde{G}}_\omega(x,\zeta)|^2\mathrm{d}\zeta,
\end{align*}
where $\|G_\omega(x,\cdot)\|_{L^2(D)}<\infty$, as shown in Lemma \ref{gnorm}, and
\begin{align}\label{eq:gseminorm} &\int_{\mathbb{R}\setminus(-\epsilon,\epsilon)}\left|\zeta\right|^{-2H-1}|\zeta\hat{\tilde{G}}_\omega(x,\zeta)|^2\mathrm{d}\zeta\notag\\ =&\int_{\mathbb{R}\setminus(-\epsilon,\epsilon)}\left|\zeta\right|^{-2H-1}\left|\zeta\int_0^x G_\omega(x,y)\mathrm{e}^{-{\rm i}\zeta y}\mathrm{d}y+\zeta\int_x^1 G_\omega(x,y)\mathrm{e}^{-{\rm i}\zeta y}\mathrm{d}y\right|^2\mathrm{d}\zeta\notag\\ =&\int_{\mathbb{R}\setminus(-\epsilon,\epsilon)}\left|\zeta\right|^{-2H-1}\bigg|\left(G_\omega(x,y)\mathrm{e}^{-{\rm i}\zeta y}\right)\Big|_{y=0}^x-\int_0^x\partial_y G_\omega(x,y)\mathrm{e}^{-{\rm i}\zeta y}\mathrm{d}y\notag\\
&+\left(G_\omega(x,y)\mathrm{e}^{-{\rm i}\zeta y}\right)\Big|_{y=x}^1-\int_x^1\partial_y G_\omega(x,y)\mathrm{e}^{-{\rm i}\zeta y}\mathrm{d}y\bigg|^2\mathrm{d}\zeta\notag\\		=&\int_{\mathbb{R}\setminus(-\epsilon,\epsilon)}\left|\zeta\right|^{-2H-1}\left|G_\omega(x,0)+\int_0^x\mathrm{e}^{-{\rm i}\zeta y}\partial_y G_\omega(x,y)\mathrm{d}y+\int_x^1\mathrm{e}^{-{\rm i}\zeta y}\partial_y G_\omega(x,y)\mathrm{d}y\right|^2\mathrm{d}\zeta\notag\\ =&\int_{\mathbb{R}\setminus(-\epsilon,\epsilon)}\left|\zeta\right|^{-2H-1}\left|G_\omega(x,0)+\widehat{\mathscr{T}G}_\omega(x,\zeta)\right|^2\mathrm{d}\zeta\notag\\ \leqslant&~\frac2H\epsilon^{-2H}|G_\omega(x,0)|^2+2\epsilon^{-2H-1}\|\mathscr T G_\omega(x,\cdot)\|_{L^2(\mathbb R)}^2<\infty
\end{align}
based on Lemma \ref{gynorm}.

Note that $C_0^\infty(\mathbb R)$ is dense in $\mathscr H^{\frac12-H}(\mathbb R)$ (cf. \cite[Theorem 7.38]{adams1975sobolev}). Therefore, for the previously claimed $\tilde G_\omega(x,\cdot)\in\mathscr H^{\frac12-H}(\mathbb R)$, there exists a sequence $\{\phi_n:=\phi_n^{x,\omega}\}_{n\in\mathbb N}\subset C_0^\infty(\mathbb R)$ converging to $\tilde G_\omega(x,\cdot)$ in the norm $\|\cdot\|{\mathscr{H}^{\frac12-H}(\mathbb{R})}$. Moreover,
according to \cite[$(2.8)$]{dasgupta1997fractional}, \eqref{fracequlformula} holds for the sequence $\left\{\int_\mathbb R\phi_n(y)\mathrm{d} B^H(y)\right\}_{n\in\mathbb N}$. Hence, we obtain
\begin{equation*}
\begin{aligned}
\mathbb{E}\left|\int_\mathbb{R} \phi_n(y) \mathrm{d} B^H(y)-\int_\mathbb{R} \phi_m(y) \mathrm{d} B^H(y)\right|^2 =&~C_H^2\int_\mathbb{R}\frac{\left|\hat{\phi}_n(\zeta)-\hat{\phi}_m(\zeta)\right|^2}{|\zeta|^{2H-1}} \mathrm{d} \zeta \\
\leqslant&~C_H^2\int_\mathbb{R}\left|\hat{\phi}_n(\zeta)-\hat{\phi}_m(\zeta)\right|^2(1+|\zeta|^{1-2H}) \mathrm{d} \zeta\\
=&~C_H^2\|\phi_n-\phi_m\|_{\mathscr H^{\frac12-H}(\mathbb{R})}^2.
\end{aligned}
\end{equation*}
As a result, the sequence $\left\{\int_\mathbb{R}\phi_n(y)\mathrm{d}B^H(y)\right\}_{n\in\mathbb N}$ converges in the mean square sense, and we define the stochastic integral $\int_{\mathbb R}\tilde G_\omega(x,y)\mathrm{d}B^H(y)$ as the mean square limit of $\int_\mathbb R\phi_n(y)\mathrm{d}B^H(y)$. Finally, we obtain
\begin{equation*}
\begin{aligned}
\mathbb{E}\left|\int_D G_\omega(x, y) \mathrm{d} B^H(y)\right|^2 &=\mathbb{E}\left|\int_\mathbb{R} \tilde{G}_\omega(x, y) \mathrm{d} B^H(y)\right|^2 =\lim_{n\rightarrow\infty}\mathbb{E}\left|\int_\mathbb{R} \phi_n(y) \mathrm{d} B^H(y)\right|^2\\ &=\lim_{n\rightarrow\infty}C_H^2\int_\mathbb{R}\frac{|\hat{\phi}_n(\zeta)|^2}{|\zeta|^{2H-1}}\mathrm{d}\zeta=C_H^2\int_\mathbb{R}\frac{|\hat{\tilde{G}}_\omega(x,\zeta)|^2}{|\zeta|^{2H-1}}\mathrm{d}\zeta,
\end{aligned}
\end{equation*}
where the last equality follows from
\begin{equation*}
\begin{aligned} &\left|\int_\mathbb{R}\frac{|\hat{\phi}_n(\zeta)|^2}{|\zeta|^{2H-1}}\mathrm{d}\zeta -\int_\mathbb{R}\frac{|\hat{\tilde{G}}_\omega(x,\zeta)|^2}{|\zeta|^{2H-1}}\mathrm{d}\zeta\right|\\
\leqslant&\int_{\mathbb R}|\hat\phi_n(\zeta)-\hat{\tilde G}_\omega(x,\zeta)|\left(|\hat\phi_n(\zeta)|+|\hat{\tilde G}_\omega(x,\zeta)|\right)|\zeta|^{1-2H}d\zeta\\
\leqslant&~\|\phi_n-\tilde G_\omega(x,\cdot)\|_{\mathscr H^{\frac12-H}(\mathbb R)}\left(\|\phi_n\|_{\mathscr H^{\frac12-H}(\mathbb R)}+\|\tilde G_\omega(x,\cdot)\|_{\mathscr H^{\frac12-H}(\mathbb R)}\right)\to0
\end{aligned}
\end{equation*}
as $n\to\infty$ due to the convergence of $\{\phi_n\}_{n\in\mathbb N}$ to $\tilde G_\omega(x,\cdot)$ in $\mathscr H^{\frac12-H}(\mathbb R)$.
\end{proof}

\begin{corollary}\label{Equidlemma}
For a given $H\in (0,1)$, the following inequality holds:
\[
\sup_{\omega\in\mathbb R}\int_D\mathbb{E}\left|\int_D G_\omega(x, y) \mathrm{d} B^H(y)\right|^2{\rm d}x \leqslant C,
\]
where $C$ is a positive constant depending only on $H$.
\end{corollary}

\begin{proof}
The result for $H=\frac12$ follows from Lemmas \ref{gnorm} and \ref{isometry}.
			
For $H\in(\frac12,1)$, utilizing Lemma \ref{isometry}, Plancherel's theorem, and the observation
\[
|\hat{\tilde G}_\omega(x,\zeta)|=\left|\int_{\mathbb R}\tilde G_\omega(x,y)\mathrm{e}^{-{\rm i}y\zeta}~\mathrm dy\right|\leqslant\int_D|G_\omega(x,y)|~\mathrm dy \leqslant\|G_\omega(x,\cdot)\|_{L^2(D)}\quad\forall\,\zeta\in\mathbb R,
\]
we have
\begin{equation}\label{eq:g1}
\begin{aligned}
\mathbb{E}\left|\int_D G_\omega(x, y) \mathrm{d} B^H(y)\right|^2 &=C_H^2\int_{(-1,1)}\frac{|\hat{\tilde{G}}_\omega(x,\zeta)|^2}{|\zeta|^{2H-1}}\mathrm{d}\zeta+C_H^2\int_{\mathbb{R}\setminus (-1,1)}\frac{|\hat{\tilde{G}}_\omega(x,\zeta)|^2}{|\zeta|^{2H-1}}\mathrm{d}\zeta\\
&\leqslant C_H^2\|G_\omega(x,\cdot)\|_{L^2(D)}^2\int_{(-1,1)}\frac1{|\zeta|^{2H-1}}~\mathrm d\zeta+C_H^2\|\hat{\tilde G}_\omega(x,\cdot)\|_{L^2(\mathbb R)}^2\\					&=\frac{C_H^2}{1-H}\|G_\omega(x,\cdot)\|_{L^2(D)}^2+2\pi C_H^2\|G_\omega(x,\cdot)\|_{L^2(D)}^2,
\end{aligned}
\end{equation}
which, together with Lemma \ref{gnorm}, implies that
\[
\sup_{\omega\in\mathbb R}\int_D\mathbb{E}\left|\int_D G_\omega(x, y) \mathrm{d} B^H(y)\right|^2{\rm d}x\leqslant C\sup_{\omega\in\mathbb R}\|G_\omega(x,y)\|_{L^2(D\times D)}^2\leqslant C.
\]

For $H\in(0,\frac12)$, it follows from Lemma \ref{isometry} and \eqref{eq:gseminorm} that
\begin{equation}\label{eq:g2}
\begin{aligned}
\mathbb{E}\left|\int_D G_\omega(x, y) \mathrm{d} B^H(y)\right|^2 &=C_H^2\int_{(-1,1)}|\zeta|^{1-2H}|\hat{\tilde{G}}_\omega(x,\zeta)|^2\mathrm{d}\zeta+C_H^2\int_{\mathbb{R}\setminus (-1,1)}|\zeta|^{-2H-1}|\zeta\hat{\tilde{G}}_\omega(x,\zeta)|^2\mathrm{d}\zeta\\
&\leqslant C_H^2\|G_\omega(x,\cdot)\|_{L^2(D)}^2+\frac2H|G_\omega(x,0)|^2+2\|\mathscr T G_\omega(x,\cdot)\|_{L^2(\mathbb R)}^2.
\end{aligned}
\end{equation}
Hence, Lemmas \ref{gnorm} and \ref{gynorm}, in conjunction with the definition of $G_\omega(\cdot,0)$, lead to
\begin{equation*}
\begin{aligned}
\sup_{\omega\in\mathbb R}\int_D\mathbb{E}\left|\int_D G_\omega(x, y) \mathrm{d} B^H(y)\right|^2{\rm d}x
\leqslant& ~C\sup_{\omega\in\mathbb R}\left(\|G_\omega\|_{L^2(D\times D)}^2+\|G_\omega(\cdot,0)\|_{L^2(D)}^2+\|\mathscr T G_\omega\|_{L^2(D\times D)}^2\right)\\
\leqslant&~ C,
\end{aligned}
\end{equation*}
which concludes the proof.
\end{proof}

\subsection{The well-posedness}
	
Utilizing the Green function $G_\omega(x,y)$, the boundary value problem \eqref{freqmaineq} has a unique mild solution in the form
\begin{equation}\label{freqmildsolution}
\hat u(x, \omega)=-\hat{f}(\omega) \int_D G_\omega(x, y) \mathrm{d} B^{H}(y),\quad\omega\in\mathbb{R},
\end{equation}
which satisfies the following regularity estimate.
	
\begin{lemma}\label{regularity}
Let $p\geqslant0$ and $f(t) \in \mathscr{H}^p\left(\mathbb{R}_{+}\right)$. The solution \eqref{freqmildsolution} of the stochastic differential equation \eqref{freqmaineq} satisfies
\begin{equation*}
\mathbb{E}\left[\int_\mathbb R\|({\rm i}\omega)^p \hat{u}(\cdot,\omega)\|_{L^2(D)}^2{\rm d}\omega\right]
\leqslant C\|f\|_{\mathscr{H}^p\left(\mathbb{R}_{+}\right)}^2,
\end{equation*}
where $C>0$ is a constant depending only on $H$.
\end{lemma}
	
\begin{proof}
By Corollary \ref{Equidlemma}, we have
\begin{align*}
\mathbb{E}\left[\int_\mathbb R\|({\rm i}\omega)^p \hat u(\cdot,\omega)\|_{L^2(D)}^2{\rm d}\omega\right]
&=\int_{\mathbb{R}}|(\mathrm{i} \omega)^p \hat{f}(\omega)|^2 \int_D \mathbb{E}\left|\int_D G_\omega(x, y) \mathrm{d} B^H(y)\right|^2 \mathrm{d} x \mathrm{d} \omega\\
&\leqslant C \int_{\mathbb{R}}|(\mathrm{i} \omega)^p \hat{f}(\omega)|^2 {\rm d} \omega \\
&\leqslant C\int_{\mathbb R}(1+|\omega|^2)^p|\hat f(\omega)|^2{\rm d}\omega=C\|f\|^2_{\mathscr{H}^p\left(\mathbb{R}_{+}\right)},
\end{align*}
which completes the proof.
\end{proof}
	
Now, we are in a position to show the well-posedness of the original problem \eqref{mainfunction} based on the equivalent problem \eqref{freqmaineq} obtained through the Fourier transform.
	
\begin{theorem}
Assuming that $f\in \mathscr{H}^2(\mathbb{R}_+)$, the direct source problem \eqref{mainfunction} has a unique solution $u\in L^2(\Omega;\mathscr H^2(\mathbb R_+;L^2(D)))$, which satisfies
\begin{equation}\label{eq:u}
\mathbb{E}\left\|u\right\|_{\mathscr H^2(\mathbb R_+;L^2(D))}^2
\leqslant C\|f\|_{\mathscr{H}^2\left(\mathbb{R}_{+}\right)}^2,
\end{equation}
where $C > 0$ is a constant depending only on $H$.
\end{theorem}

\begin{proof}
The proof draws inspiration from \cite{gong2021numerical}. Let $\{f(t)\}_{t\in\mathbb R}$ be the zero extension of $\{f(t)\}_{t\in\mathbb R_+}$, as initially explained in Section \ref{s3.1}. For any $x\in D$ and $t\in\mathbb{R}$, with $\hat u(x,\omega)$ considered as the mild solution of \eqref{freqmaineq}, we define the inverse Fourier transform of $\hat u(x,\omega)$ as follows: 
\[
\breve{u}(x,t):=-\int_{-\infty}^t f(\tau) \mathscr{F}^{-1}\left[\int_D G_\omega(x, y) \mathrm{d} B^H(y)\right](t-\tau) \mathrm{d} \tau. 
\]
By Plancherel's theorem and Lemma \ref{regularity}, we obtain $\partial_t\breve u,\partial_t^2\breve u\in L^2(\Omega;L^2(D\times \mathbb R))$ and
\begin{equation*}
\begin{aligned}
\mathbb E\|\breve u\|_{\mathscr H^2(\mathbb R;L^2(D))}^2=&~\mathbb E\left[\int_{\mathbb R}(1+|\omega|^2)^2\|\hat u(\cdot,\omega)\|_{L^2(D)}^2{\rm d}\omega\right]\\
\leqslant&~\mathbb E\left[\int_{\mathbb R}\|\hat u(\cdot,\omega)\|_{L^2(D)}^2{\rm d}\omega\right]+2\mathbb E\left[\int_{\mathbb R}\|({\rm i}\omega)\hat u(\cdot,\omega)\|_{L^2(D)}^2{\rm d}\omega\right]\\
&+\mathbb E\left[\int_{\mathbb R}\|({\rm i}\omega)^2 \hat u(\cdot,\omega)\|_{L^2(D)}^2{\rm d}\omega\right]\\
\leqslant&~C\|f\|_{\mathscr H^2(\mathbb R_+)}^2,
\end{aligned}
\end{equation*}
which also implies that the Caputo fractional derivative of $\breve{u}$ with respect to the time $t$ is properly defined.

Let $u(x, t)$ be the restriction of $\breve{u}(x, t)$ to $t$ belonging to the set of non-negative real numbers, i.e., 
\[
u(x, t):=\breve{u}(x, t)|_{t\in\mathbb R_+}. 
\]
It can be readily verified that, in a mean square sense, the function $u$ defined as described above is the unique mild solution of \eqref{mainfunction}. It is clear to note that
\[
u(x,0)=\breve{u}(x,0)=0,\quad \partial_t u(x,0)=\partial_t\breve{u}(x,0)=0.
\]
In addition, it also satisfies \eqref{eq:u}.
\end{proof}

\section{The inverse problem}\label{sec:4}

In this section, our primary focus is on addressing the uniqueness and instability in the reconstruction of the phaseless Fourier mode $|\hat f(\omega)|$ of the source function $f$ from the measured data $\{u(0,t)\}_{t\geq 0}$. To subsequently recover $|f(t)|$ from $|\hat f(\omega)|$, commonly referred to as the phase retrieval problem, we  introduce and employ the PhaseLift technique.

Evaluating \eqref{freqmildsolution} at $x=0$ and then taking the expected value and variance on both sides, we deduce 
\begin{equation}\label{inverseformula}
\mathbb{E}[\hat u(0, \omega)]=0, \quad \mathbb{V}[\hat u(0, \omega)]=R(\omega)|\hat{f}(\omega)|^2 ,
\end{equation}
where $R(\omega)$ is a critical constant depending on $\omega$ and is given by
\[
 R(\omega)=\mathbb{E}\left|\int_D G_\omega(0,y) \mathrm{d} B^H(y)\right|^2.
\]
Here, for any $y\in D$, we have
\begin{equation*}
G_\omega(0,y) =\left\{
\begin{aligned}
&y-1,\quad  &\omega=0,\\			&\frac{\mathrm{e}^{\sqrt{s}y}-\mathrm{e}^{\sqrt{s}(2-y)}}{\sqrt{s}\left(1+\mathrm{e}^{2 \sqrt{s}}\right)}, \quad& \omega\neq0.
\end{aligned}
\right.
\end{equation*}

\subsection{Uniqueness}
\label{sec:4.1}

It is clear to note from \eqref{inverseformula} that $|\hat f(\omega)|$ can be uniquely determined by
\begin{equation}\label{Fhat}
|\hat{f}(\omega)|=\left(\frac{\mathbb{V}[\hat u(0,\omega)]}{R(\omega)}\right)^\frac{1}{2},\quad\omega\in\mathbb{R}
\end{equation}
if $R(\omega)>0$. In fact, as stated in the following lemma, the uniqueness mentioned here is established.

\begin{lemma}\label{lm:unique}
For $H\in(0,1)$, it holds for any $\omega\in\mathbb R$ that $R(\omega)>0$.
\end{lemma}

\begin{proof}
We have from Lemma \ref{isometry} that
\begin{equation*}
R(\omega)=C_H^2\int_\mathbb{R}\frac{|\hat{\tilde{G}}_\omega(0,\zeta)|^2}{|\zeta|^{2H-1}}\mathrm{d}\zeta
\geqslant C_H^2\int_0^1\frac{|\hat{\tilde{G}}_\omega(0,\zeta)|^2}{\zeta^{2H-1}}\mathrm{d}\zeta,
\end{equation*}
where $\hat{\tilde G}_\omega(0,\zeta)$ with $\zeta\in(0,1)$ can be calculated as follows.
		
If $\omega=0$, it holds
\begin{equation*}
\begin{aligned}
\left|\hat{\tilde{G}}_0(0,\zeta)\right|^2
&=\left|\int_{D}G_0(0,y)\mathrm{e}^{-\mathrm{i}\zeta y}\mathrm{d}y\right|^2				=\left|\int_0^1(y-1)\mathrm{e}^{-\mathrm{i}\zeta y}\mathrm{d}y\right|^2\\				&=\left|\frac{\mathrm{e}^{-\mathrm{i}\zeta}-1+{\rm i}\zeta}{\zeta^2}\right|^2				=\frac{(\cos(\zeta)-1)^2+(\sin(\zeta)-\zeta)^2}{\zeta^4},
\end{aligned}
\end{equation*}
which, together with $\frac{(\cos(\zeta)-1)^2+(\sin(\zeta)-\zeta)^2}{\zeta^{2H+3}}>0$ for any $\zeta\in(0,1)$, implies that
\begin{equation*}
\begin{aligned}
R(0)&\geqslant C_H^2\int_0^1\frac{(\cos(\zeta)-1)^2+(\sin(\zeta)-\zeta)^2}{\zeta^{2H+3}}\mathrm{d}\zeta>0.
\end{aligned}
\end{equation*}

If $\omega\neq0$, it follows from a straightforward calculation that 
\begin{equation*}
\begin{aligned}
\hat{\tilde{G}}_\omega(0,\zeta)
&=\int_{D}G_\omega(0,y)e^{-\mathrm{i}\zeta y}\mathrm{d}y\\
&=\frac{1}{\sqrt{s}\left(1+\mathrm{e}^{2\sqrt{s}}\right)}\left[\frac{\mathrm{e}^{\sqrt s-\mathrm{i}\zeta}-1}{\sqrt{s}-\mathrm{i}\zeta}+\frac{\mathrm{e}^{\sqrt s-\mathrm{i}\zeta}-\mathrm{e}^{2\sqrt{s}}}{\sqrt{s}+\mathrm{i}\zeta}\right]\\
&=\frac{2\sqrt{s}\mathrm e^{\sqrt s-{\rm i}\zeta}-\sqrt{s}(\mathrm e^{2\sqrt{s}}+1)+{\rm i}\zeta(\mathrm e^{2\sqrt s}-1)}{\sqrt s(1+\mathrm e^{2\sqrt{s}})(s+\zeta^2)}.
\end{aligned}
\end{equation*}
For $\omega\neq0$, i.e., $s\neq0$, we assert that $|\hat{\tilde G}_{\omega}(0,\zeta)|\not\equiv0$ for $\zeta\in[0,1]$. In fact, if $\sqrt s\neq2n\pi\mathrm{i}$ with $n\in \mathbb{Z}\setminus\{0\}$, then
\[ |\hat{\tilde{G}}_\omega(0,0)|=\left|\frac{2\mathrm{e}^{\sqrt s}-\mathrm{e}^{2\sqrt s}-1}{s\left(1+\mathrm{e}^{2 \sqrt{s}}\right)}\right|=\left|\frac{(\mathrm e^{\sqrt s}-1)^2}{s\left(1+\mathrm{e}^{2 \sqrt{s}}\right)}\right|>0.
\]
If $\sqrt{s}=2n\pi\mathrm{i}$ with $n\in\mathbb{Z}\setminus\{0\}$, then ${\rm e}^{\sqrt s}=1$ and
\[
|\hat{\tilde{G}}_\omega(0,1)|=\left|\frac{\rm e^{-\rm i}-1}{1-4n^2\pi^2}\right|>0,
\]
which finishes the assertion. As a result,
\begin{equation*}
R(\omega) \geqslant C_H^2\int_0^1\frac{|\hat{\tilde{G}}_\omega(0,\zeta)|^2}{\zeta^{2H-1}}\mathrm{d}\zeta>0
\end{equation*}
due to the continuity of $\hat{\tilde G}_\omega(0,\cdot)$ in $[0,1]$ for $\omega\neq0$.
\end{proof}

\subsection{Instability}
\label{sec:4.2}

While the uniqueness of the reconstruction for $|\hat f(\omega)|$ is confirmed by \eqref{Fhat} and Lemma \ref{lm:unique}, the recovery process is found to be unstable, as demonstrated in the following theorem. In this subsection, we always assume that $|\omega|>1$.
	
\begin{theorem}\label{instability}
For $H\in[\frac12,1)$, there exists a constant $C>0$ independent of $\omega$ such that
\[
R(\omega)\leqslant C|\omega|^{-\alpha_{\max}}£¬
\]
where $\alpha_{\max}$ is defined in Lemma \ref{lm:s}. For $H\in(0,\frac12)$, assuming additionally that $\alpha_n<2$, the following holds:
\[
\lim_{|\omega|\to\infty}R(\omega)=0.
\]
\end{theorem}

\begin{proof}
For $H=\frac12$, we deduce from Lemma \ref{gnorm} that
\[
R(\omega)=\|G_\omega(0,\cdot)\|_{L^2(D)}^2\leqslant C|s|^{-1}.
\]
For $H\in(\frac12,1)$, by utilizing \eqref{eq:g1} and Lemma \ref{gnorm}, we obtain
\begin{equation*}
\begin{aligned}
R(\omega) &\leqslant\frac{C_H^2}{1-H}\|G_\omega(0,\cdot)\|_{L^2(D)}^2+2\pi C_H^2\|G_\omega(0,\cdot)\|_{L^2(D)}^2\leqslant C|s|^{-1}.
\end{aligned}
\end{equation*}
Then, the result for the case $H\in[\frac12,1)$ follows directly from the fact that $|s|\geqslant\sin\left(\frac{\pi\alpha_{\max}}2\right)|\omega|^{\alpha_{\max}}$, as given in Lemma \ref{lm:s}.

For $H\in(0,\frac12)$, the estimate \eqref{eq:g2} gives
\begin{equation*}
\begin{aligned}
R(\omega)&\leqslant C_H^2\|G_\omega(0,\cdot)\|_{L^2(D)}^2+\frac2H|G_\omega(0,0)|^2+2\|\mathscr T G_\omega(0,\cdot)\|_{L^2(\mathbb R)}^2,
\end{aligned}
\end{equation*}
where $\|G_\omega(0,\cdot)\|_{L^2(D)}^2\leqslant C|s|^{-1}$,
\[
|G_\omega(0,0)|^2=\left|\frac{1-{\rm e}^{2\sqrt s}}{\sqrt s(1+{\rm e}^{2\sqrt s})}\right|^2\leqslant|s|^{-1},
\]
and
\[
\|\mathscr T G_\omega(0,\cdot)\|_{L^2(\mathbb R)}^2\leqslant h\left(\Re[\sqrt s]\right)\leqslant 2\left|\Re[\sqrt s]\right|^{-1},\quad\Re[\sqrt s]\to\infty
\]
according to \eqref{eq:hbound} and \eqref{eq:Tg}.
It then suffices to estimate $\left|\Re[\sqrt s]\right|^{-1}$.
Note that
\[
\left|\Re[\sqrt s]\right|^{-1}=|s|^{-\frac12}\left|\cos\left(\frac{\arg(s)}2\right)\right|^{-1}
\]
and
\[
\left|\cos\left(\frac{\arg(s)}2\right)\right|=\sqrt{\frac{\cos\left(\arg(s)\right) + 1}{2}},
\]
where
\[
\cos(\arg(s))=\frac{\Re{[s]}}{|s|}=\frac{\sum\limits_{k=1}^{n}|\omega|^{\alpha_{k}}\cos\left(\frac{\pi\alpha_k}{2}\right)}{\sqrt{\sum\limits_{k=1}^{n}|\omega|^{2\alpha_{k}} + 2\sum\limits_{1\leqslant i<j\leqslant n}|\omega|^{\alpha_{i}+\alpha_{j}}\cos\left(\frac{\pi\alpha_i-\pi\alpha_j}{2}\right)}}\to\cos\frac{\pi\alpha_n}{2},\quad |\omega|\to\infty.
\]
We then get
\[
R(\omega)\leqslant C|s|^{-1}+C|s|^{-\frac12}\left(\frac{\cos\left(\arg(s)\right) + 1}{2}\right)^{-\frac12}\to0,\quad|\omega|\to\infty,
\]
which completes the proof.
\end{proof}

Theorem \ref{instability} implies that the reconstruction for $|\hat{f}(\omega)|$ using \eqref{Fhat} is unstable. More precisely, any small perturbation in the data $\mathbb{V}[\hat u(0,\omega)]$ will be significantly amplified in the reconstruction when $|\omega|$ is sufficiently large. The degree of ill-posedness follows a polynomial form of $|\omega|^{-\gamma}$, where
\begin{equation*}
\gamma=
\left\{
\begin{aligned}
& \alpha_{\max}, \quad&H\in[\frac12,1),\\
& \frac{\alpha_{\max}}{2}, \quad& H\in(0,\frac12).\\
\end{aligned}
\right.
\end{equation*}
Furthermore, for $H\in(0,\frac12)$, if $\alpha_n=2$, which is not covered by Theorem \ref{instability}, the limit behavior of $R(\omega)$ remains uncertain due to the facts that
\[
\lim_{|\omega|\to\infty}\cos(\arg(s))=-1
\]
and thus
\[
\lim_{|\omega|\to\infty}\cos\left(\frac{\arg(s)}{2}\right)=0,
\]
which makes the limit behavior of
\[
\left|\Re[\sqrt s]\right|^{-1}=|s|^{-\frac12}\left|\cos\left(\frac{\arg(s)}2\right)\right|^{-1}
\]
unclear as $|\omega|\to\infty$.

\section{Numerical experiments}\label{sec:5}

This section is dedicated to the numerical solutions of the direct and inverse problems for the two-term time-fractional stochastic diffusion-wave equation
\begin{equation}\label{numericeq}
\left\{
\begin{array}{lll}			\partial_{t}^{\alpha_1}u(x,t)+\partial_{t}^{\alpha_2}u(x,t)-\partial_{xx} u(x,t)=f(t)\dot{B}^H(x), &(x,t)\in D\times (0,T],\\
u(x,0)=0, & x\in D, \\
\partial_tu(x,0)=0, & x\in D,\quad \text{if}\ \  \alpha_2\in(1,2),\\
\partial_{x}u(0,t)=0, \ u(1,t)=0,& t\in [0,T],
\end{array}
\right.
\end{equation}
where $D=(0, 1)$, $T>0$, and $\alpha_i\in(0,2)$ for $i=1,2$ with $\alpha_1<\alpha_2$. To simplify notation, we use the vector $\bm \alpha:=[\alpha_1,\alpha_2]$.

\subsection{The direct problem}

To generate synthetic data, we employ the finite difference method (FDM) presented in \cite[Sections 2.3 and 3.1]{sun2020fractional} to discretize the direct problem \eqref{numericeq}. Specifically, we start by discretizing the temporal and spatial intervals into $N$ and $M$ subintervals with nodes as follows:
\[
t_n=n\tau,~n=0,1,\cdots,N, \quad x_m=mh,~m=0,1,\cdots,M,
\]
where  $\tau=T/N$ and $h=1/M$. We also define the increment and the variation of the fractional Brownian motion $B^H$ as
\[
\delta B^H_m:=B^H(x_{m+1})-B^H(x_m),\quad\delta\dot B^H_m:=\frac{\delta B^H_m}h.
\]

For $0<\alpha_1<\alpha_2<1$, the numerical scheme for solving \eqref{numericeq} reads (cf. \cite[Section 2.3]{sun2020fractional}):
\begin{equation*}
\left\{	
\begin{aligned}
&\delta_{t}^{\alpha_1}u_m^n+\delta_{t}^{\alpha_2}u_m^n-\delta_{x}^2u_m^n=f(t_n)\delta\dot B^H_m,\\
&u_m^0=0,\\
&\delta_xu_0^n=0,~u_M^n=0,
\end{aligned}
\right.
\end{equation*}
where $m=1, \dots, M-1, n=1, \dots, N$, $u_m^n$ is an approximation of the exact solution $u(x_m,t_n)$, and the difference operators $\delta_x$, $\delta_x^2$, and $\delta_t^{\alpha}$ are defined as
\begin{align*}
\delta_xu_m^n:&=\frac{u^n_{m+1}-u^n_{m-1}}{2h},\\
\delta_x^2u_m^n:&=\frac{u_{m-1}^n-2u_m^n+u_{m+1}^n}{h^2},\\
\delta_{t}^{\alpha}u_m^n:&=\frac1{\tau^{\alpha}\Gamma(2-\alpha)}\left[a_0^{(\alpha)}u^n_m-\sum\limits_{k=1}^{n-1}\left(a_{n-k-1}^{(\alpha)}-a_{n-k}^{(\alpha)}\right)u^k_m\right],\quad 0<\alpha<1
\end{align*}
with $a_l^{(\alpha)}:=(l+1)^{1-\alpha}-l^{1-\alpha}$ for $l\geqslant0$.
To avoid confusion, we mention that $\delta_x^2=\delta_x^-\delta_x^+$ is the combination of the forward and backward difference operators $\delta_x^+$ and $\delta_x^-$, rather than the combination $\delta_x\circ\delta_x$.

For $0<\alpha_1<1<\alpha_2<2$, the numerical scheme becomes an average on two adjacent levels (cf. \cite[Section 3.1]{sun2020fractional}):
\begin{equation*}
\left\{	
\begin{aligned}
&\frac{\delta_{t}^{\alpha_1}u_m^n+\delta_{t}^{\alpha_1}u_m^{n-1}}2+\delta_{t}^{\alpha_2}u_m^{n-\frac12}-\frac{\delta_{x}^2u_m^n+\delta_{x}^2u_m^{n-1}}2=f_{n-\frac12}\delta\dot B^H_m,\\
&u_m^0=0,\\
&\delta_xu_0^n=0,~u_M^n=0,
\end{aligned}
\right.
\end{equation*}
and, for $1<\alpha_1<\alpha_2<2$, the numerical scheme is
\begin{equation*}
\left\{	
\begin{aligned}
&\delta_{t}^{\alpha_1}u_m^{n-\frac12}+\delta_{t}^{\alpha_2}u_m^{n-\frac12}-\frac{\delta_{x}^2u_m^n+\delta_{x}^2u_m^{n-1}}2=f_{n-\frac12}\delta\dot B^H_m,\\
&u_m^0=0,\\
&\delta_xu_0^n=0,~u_M^n=0,
\end{aligned}
\right.
\end{equation*}
where $m=1, \dots, M-1, n=1, \dots, N$, 
\begin{align*}
f_{n-\frac12}:&=\frac{f(t_n)+f(t_{n-1})}2,\\
\partial_{t}^{\alpha}u_m^{n-\frac12}:&=\frac1{\tau^{\alpha}\Gamma(3-\alpha)}\left[b_0^{(\alpha)}(u^n_m-u^{n-1}_m)-\sum\limits_{k=1}^{n-1}\left(b_{n-k-1}^{(\alpha)}-b_{n-k}^{(\alpha)}\right)(u^k_m-u^{k-1}_m)\right],\quad 1<\alpha<2
\end{align*}
with $b_l^{(\alpha)}:=(l+1)^{2-\alpha_i}-l^{2-\alpha}$ for $l\geqslant0.$

The corresponding compact matrix form of the FDM can be summarized as follows:
\begin{equation}\label{FDM}
\left[\begin{array}{cccccc}
\beta & -2 & 0 &\ldots & 0&0 \\
-1 & \beta & -1 & \ldots & 0&0 \\
0 & -1 & \beta & \ldots & 0 &0\\
\vdots & \ddots & \ddots & \ddots & \vdots &\vdots\\
0 & \cdots & 0 & -1 & \beta & -1 \\
0 & \cdots & 0 & 0 & -1 & \beta
\end{array}\right]
\left[\begin{array}{c}
u_0^n\\
u_1^n \\
u_2^n \\
\vdots \\
\vdots \\
u_{M-1}^n
\end{array}\right]
=\left[\begin{array}{c}
w_0^n \\
w_1^n \\
w_2^n \\
\vdots \\
\vdots \\
w_{M-1}^n
\end{array}\right],
\end{equation}
where $n=1, \dots, N$, the diagonal entry $\beta$ and $w_{m}^n$ vary depending on the specific cases of $\alpha_j$, with $j=1,2$.

For $0<\alpha_1<\alpha_2<1$, the values of $\beta$ and $w_{m}^n$ are
\begin{align*}
\beta=&~2+\frac{h^2a_0^{(\alpha_1)}}{\tau^{\alpha_1}\Gamma(2-\alpha_1)}+\frac{h^2a_0^{(\alpha_2)}}{\tau^{\alpha_2}\Gamma(2-\alpha_2)},\\
w_m^n=&\sum\limits_{j=1}^2\sum\limits_{k=1}^{n-1}\frac{h^2}{\tau^{\alpha_j}\Gamma(2-\alpha_{j})}\left(a_{n-k-1}^{(\alpha_j)}-a_{n-k}^{(\alpha_j)}\right)u^k_m+h f(t_n)\delta B^H_m.
\end{align*}
For $0<\alpha_1<1<\alpha_2<2$, the values of $\beta$ and $w_{m}^n$ are given by
\begin{align*}
\beta=&~2+\frac{h^2a_0^{(\alpha_1)}}{\tau^{\alpha_1}\Gamma(2-\alpha_1)}+\frac{2h^2b_0^{(\alpha_2)}}{\tau^{\alpha_2}\Gamma(3-\alpha_2)},\\
w_m^n=&-\frac{h^2a_{1}^{(\alpha_1)}}{\tau^{\alpha_1}\Gamma(2-\alpha_{1})}u^{n-1}_m+\sum\limits_{k=1}^{n-2}\frac{h^2}{\tau^{\alpha_1}\Gamma(2-\alpha_{1})}\left(a_{n-k-2}^{(\alpha_1)}-a_{n-k}^{(\alpha_1)}\right)u^k_m+u^{n-1}_{m-1}-2u^{n-1}_{m}+u^{n-1}_{m+1}\\
&+\frac{2h^2}{\tau^{\alpha_2}\Gamma(3-\alpha_{2})}\left[b_{0}^{(\alpha_2)}u^{n-1}_m+\sum\limits_{k=1}^{n-1}\left(b_{n-k-1}^{(\alpha_2)}-b_{n-k}^{(\alpha_2)}\right)\left(u^k_m-u^{k-1}_m\right)\right]
+2h f_{n-\frac12}\delta B^H_m.
\end{align*}
For $1<\alpha_1<\alpha_2<2$, the values are
\begin{align*}
\beta=&~2+\frac{2h^2b_0^{(\alpha_1)}}{\tau^{\alpha_1}\Gamma(3-\alpha_1)}+\frac{2h^2b_0^{(\alpha_2)}}{\tau^{\alpha_2}\Gamma(3-\alpha_2)},\\
w_m^n=&\sum\limits_{j=1}^2\frac{2h^2}{\tau^{\alpha_j}\Gamma(3-\alpha_{j})}\left[b_{0}^{(\alpha_j)}u^{n-1}_m+\sum\limits_{k=1}^{n-1}\left(b_{n-k-1}^{(\alpha_j)}-b_{n-k}^{(\alpha_j)}\right)\left(u^k_m-u^{k-1}_m\right)\right]\\
&+u^{n-1}_{m-1}-2u^{n-1}_{m}+u^{n-1}_{m+1}
+2h f_{n-\frac12}\delta B^H_m.
\end{align*}

Based on the schemes described above, we obtain the numerical solution denoted as $u_0^n$, which serves as an approximation of the exact solution $u(0,t_n)$. The error estimate between the numerical solution $u_0^n$ and the exact solution $u(0,t_n)$ can be explored using a procedure analogous to the one employed in the deterministic case as presented in \cite{sun2020fractional}. However, this error analysis is outside the scope of the current work and, therefore, is not included here.

To further generate synthetic data based on the numerical solution $u_0^n$, we consider the incorporation of noise. Recognizing that observed data in practical scenarios are often subject to contamination from various sources, we introduce the following noisy data model:
\begin{align}\label{adnoisy}
u^{n,\epsilon}_0=u_0^n\left(1+\epsilon\eta_n\right),\quad n=0,\dots,N,
\end{align}
where $\epsilon > 0$ represents the noise level, and $\{\eta_n\}_{n=0,\ldots,N}$ is a sequence of independent random variables uniformly distributed between $-1$ and 1. The required data, denoted as
\[
\hat u_0^{n_{\omega}, \epsilon},\quad n_{\omega}=1,\cdots,N_{\omega}
\]
can be generated by performing a discrete Fourier transform on the noisy data $\{u^{n,\epsilon}_0\}_{n=0,\cdots,N}$ at specific  discrete frequencies $\{\omega_{n_{\omega}}\}_{n_{\omega}=1}^{N_{\omega}}$. The details of the frequency selection process will be presented in the following section.

\subsection{The inverse problem}

In this section, we present the numerical reconstruction of $\{|f(t_n)|\}_{n=0}^N$ at discrete points $\{t_n\}_{n=0}^N$. To achieve this, two steps are required. First, the phaseless Fourier modes ${|\hat f^\epsilon(\omega_{n_{\omega}})|}_{n_{\omega}=1}^{N_{\omega}}$ is obtained from the noisy data $\{\hat u_0^{n_{\omega},\epsilon}\}_{n_{\omega}=1}^{N_{\omega}}$ using \eqref{Fhat}, combined with a regularization technique. Second, the numerical approximation of ${|f(t_n)|}_{n=0}^N$ is reconstructed from $\{|\hat f^\epsilon(\omega_{n_{\omega}})|\}_{n_{\omega}=1}^{N_{\omega}}$ using the PhaseLift method.

\subsubsection{Spectral cut-off regularization}	

For the inverse problem, it is shown in Section \ref{sec:4.1} that $|\hat f(\omega)|$ can be uniquely determined through \eqref{Fhat}. Nevertheless, the reconstruction is characterized as unstable, as elaborated in Section \ref{sec:4.2}.
Consequently, in this context, a spectral cut-off regularization is employed when computing $\{|\hat f^\epsilon(\omega_{n_{\omega}})|\}_{n_{\omega}=1}^{N_{\omega}}$ from the noisy data $\{\hat u_0^{n_{\omega},\epsilon}\}_{n_{\omega}=1}^{N_{\omega}}$ using the formula
\begin{align}\label{hatFnoisy}
|\hat{f}^\epsilon(\omega_{n_{\omega}})|=\left(\frac{\mathbb{V}[\hat u_0^{n_{\omega}, \epsilon}]}{R(\omega_{n_\omega})}\right)^\frac{1}{2},\quad n_{\omega}=1,\cdots,N_{\omega}.
\end{align}
To address this ill-posed problem more effectively, we opt for $\omega_{n_{\omega}}=\text{linspace}(0,W,n_{\omega})$, where $W$ is designated as a regularization parameter. This choice removes high-frequency modes with $\omega>W$ from the noisy data.

Note that the second moment of the stochastic integral $R(\omega)$ involved in \eqref{Fhat} is independent of the data and can be computed in advance. Figure \ref{denH357} presents its values concerning $\omega$ for different values of $H=0.3,0.5,0.7$. The graph illustrates that for a fixed $H$ (resp. $\bm{\alpha}$), the value of $R(\omega)$ decreases more rapidly when $\alpha_i$ (resp. $H$) is larger. It is evident that the choice of the regularization parameter $W$ plays an essential role in the reconstruction, and its selection will be detailed in the subsequent numerical examples.

\begin{figure}[h]
  \centering
  \includegraphics[width=0.31\textwidth]{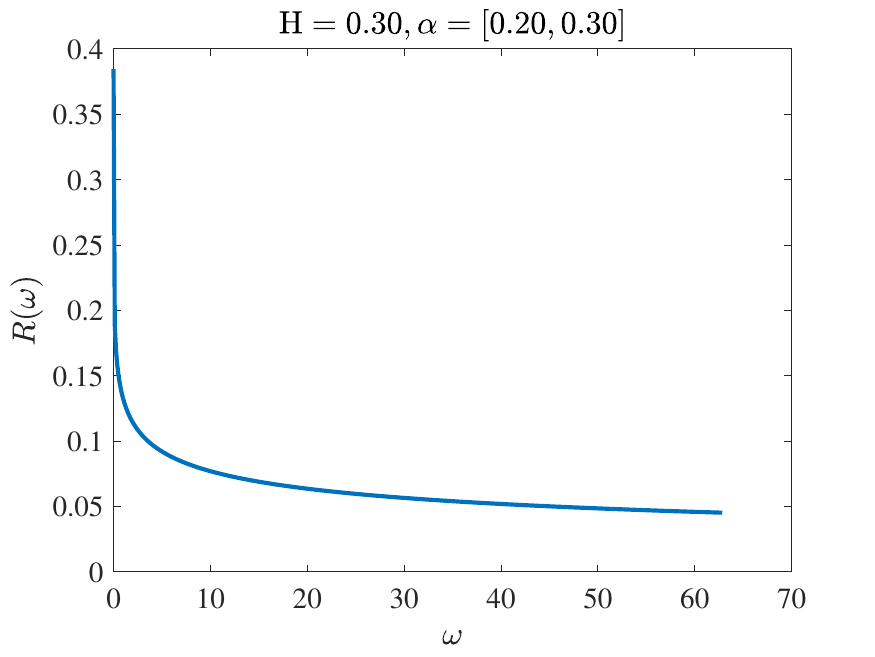}
  \includegraphics[width=0.31\textwidth]{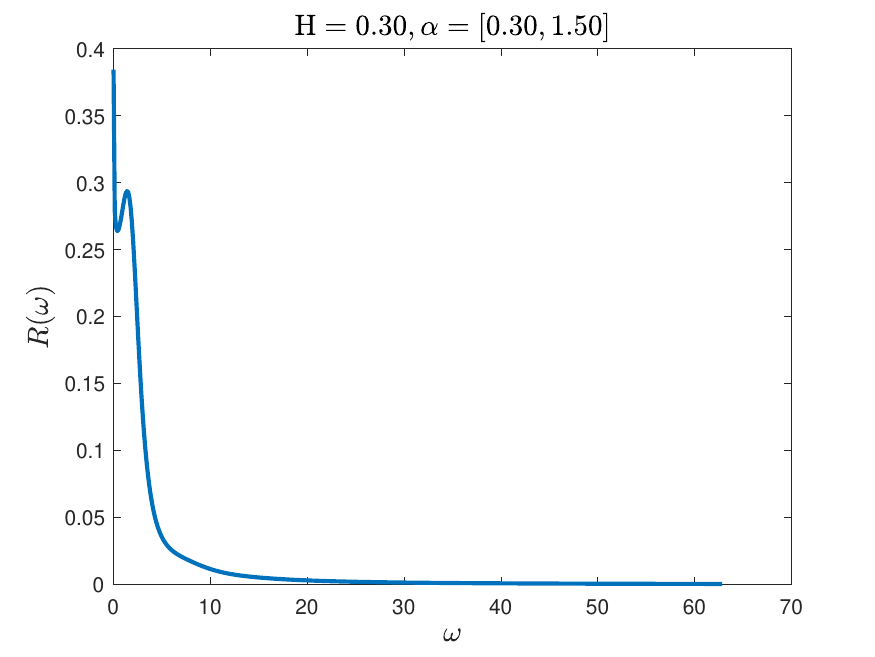}
  \includegraphics[width=0.31\textwidth]{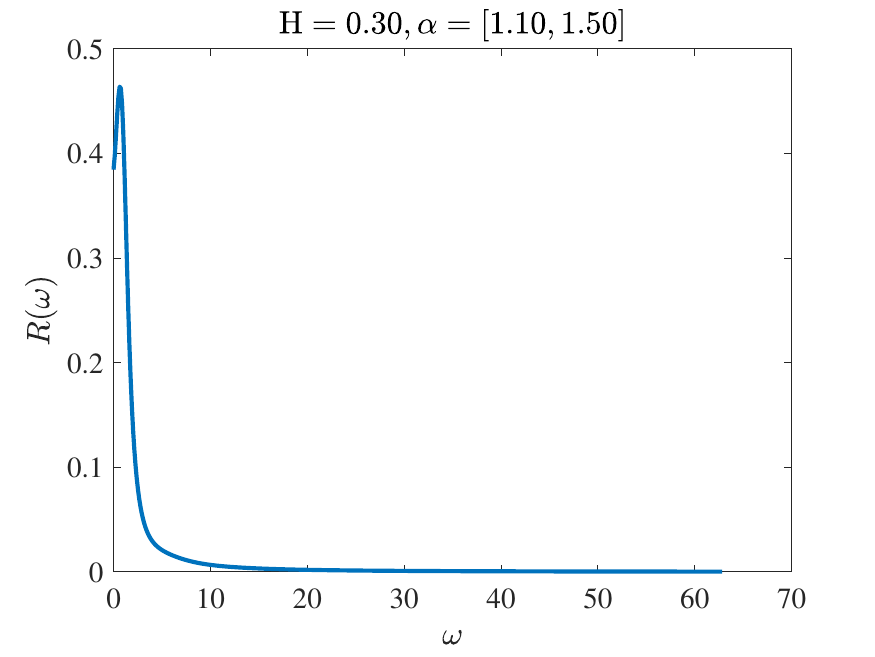}\\
  \includegraphics[width=0.31\textwidth]{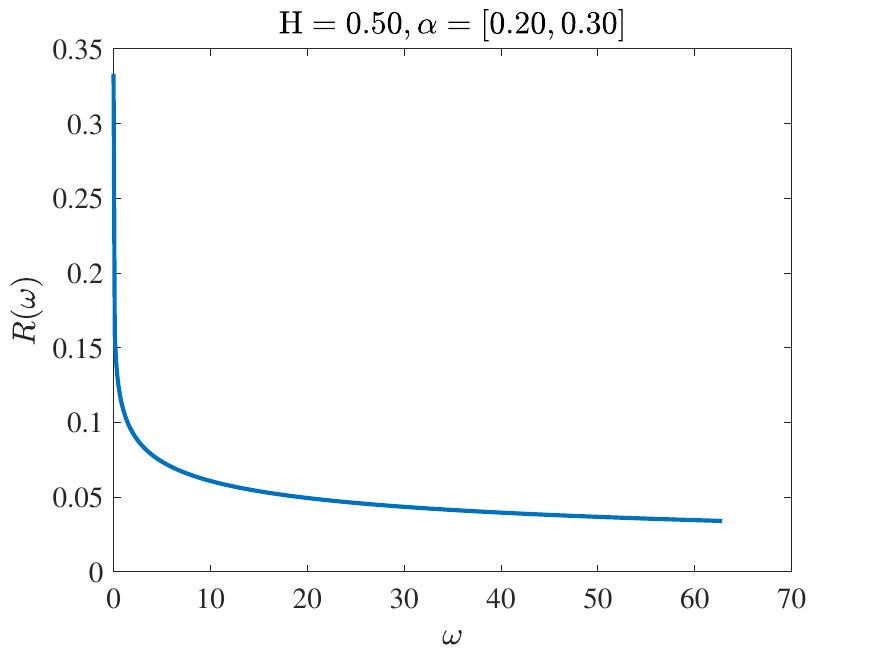}
  \includegraphics[width=0.31\textwidth]{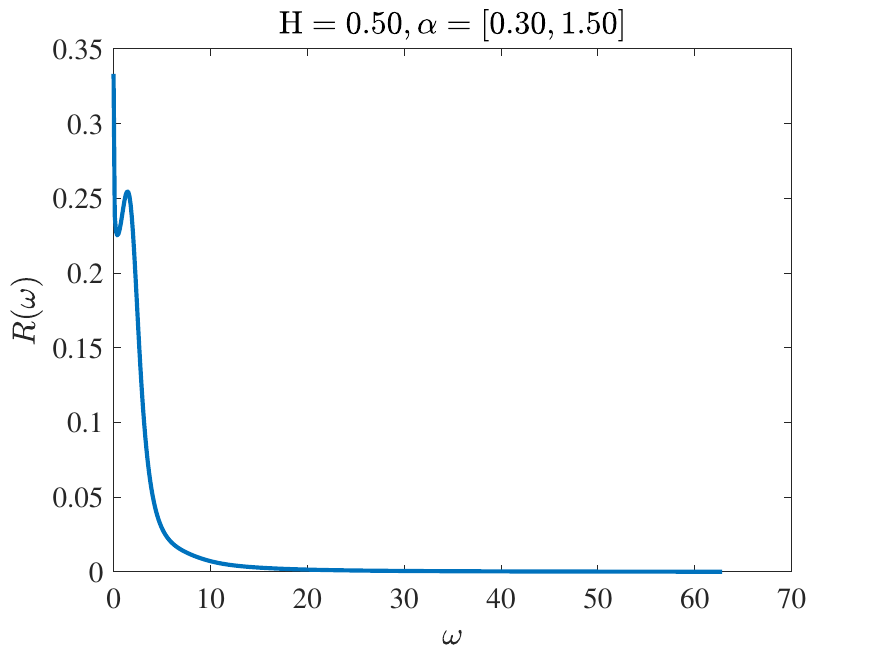}
  \includegraphics[width=0.31\textwidth]{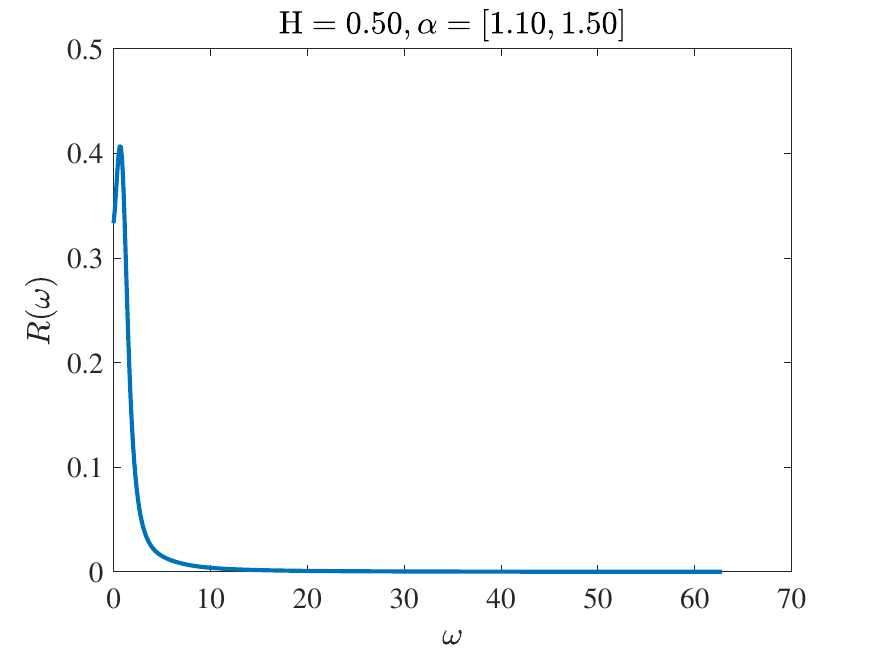}\\
  \includegraphics[width=0.31\textwidth]{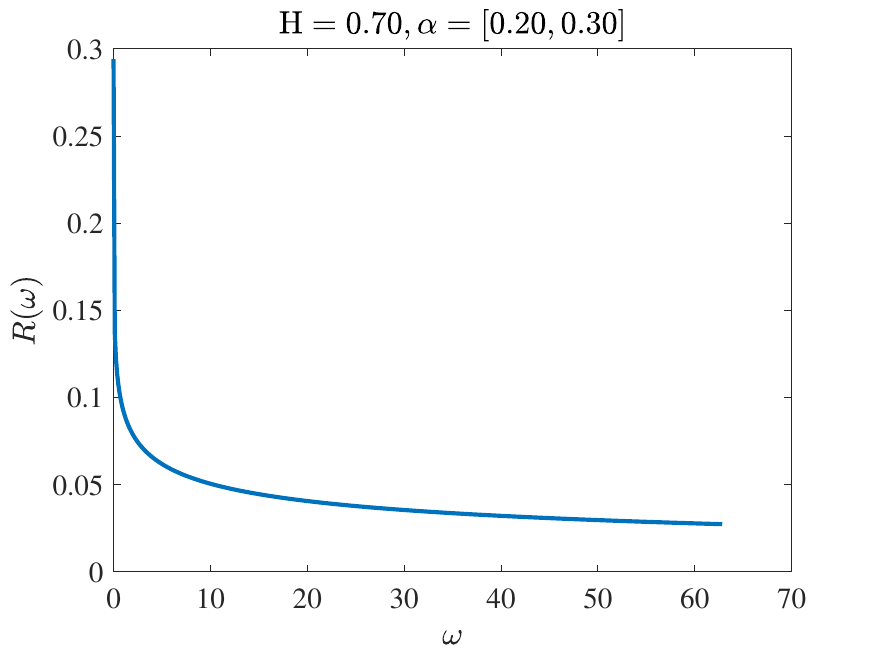}
  \includegraphics[width=0.31\textwidth]{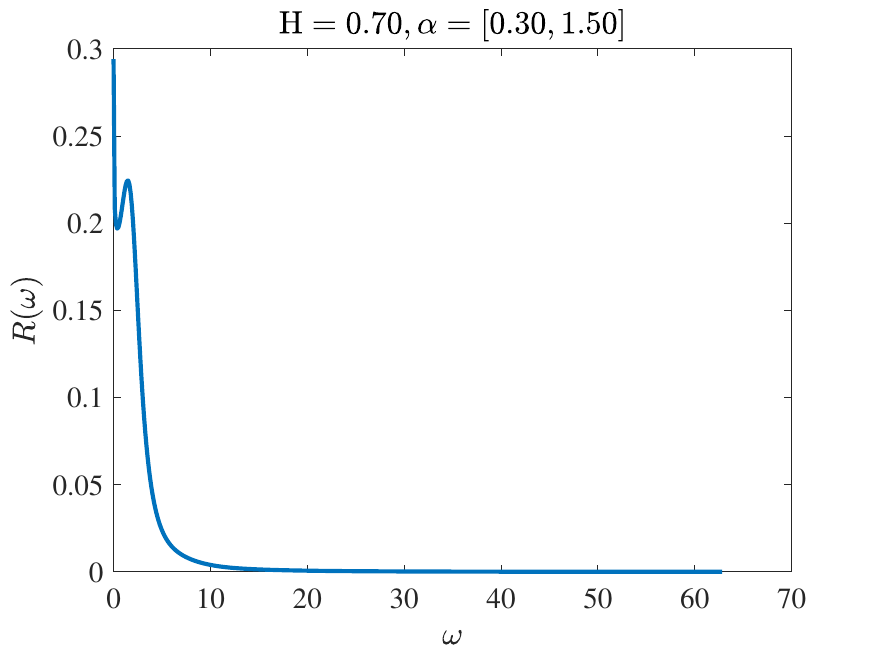}
  \includegraphics[width=0.31\textwidth]{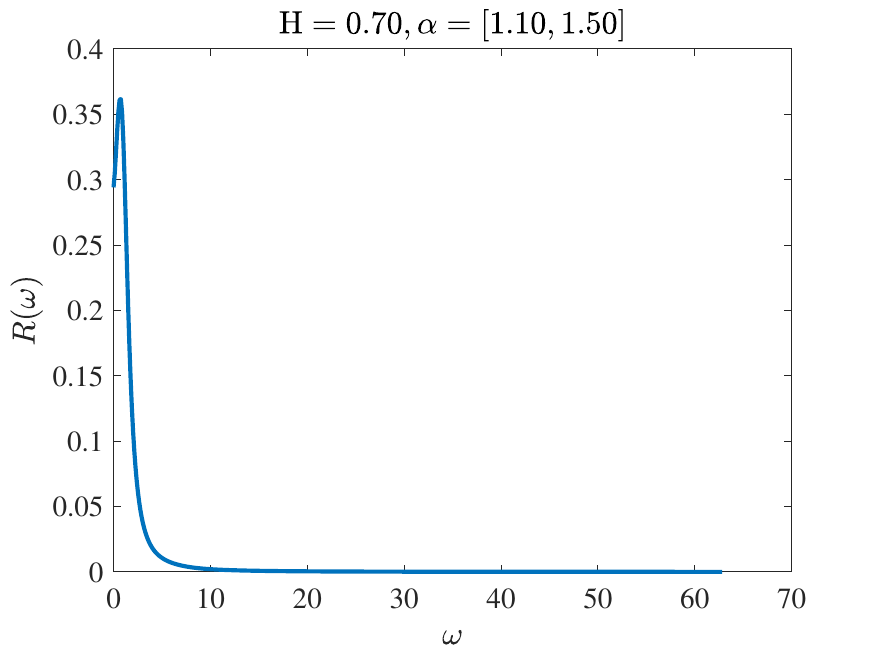}
  \caption{The values of $R(\omega)$ for (top) $H=0.3$, (middle) $H=0.5$, and (bottom) $H=0.7$ with different $\bm{\alpha}$.}\label{denH357}
\end{figure}

\subsubsection{PhaseLift algorithm}

Based on the phaseless Fourier modes $\{|\hat{f}^\epsilon(\omega_{n_{\omega}})|\}_{n_\omega=1}^{N_{\omega}}$ obtained earlier,  our next objective is to obtain the approximation $\{|f^{\epsilon}(t_n)|\}_{n=0}^N$ of $\{|f(t_n)|\}_{n=0}^N$ from $\{|\hat{f}^\epsilon(\omega_{n_{\omega}})|\}_{n_\omega=1}^{N_{\omega}}$. This problem, which involves reconstructing the signal at discrete points from the magnitude of its discrete Fourier transform, is known as the discrete phase retrieval problem \cite{harker1948phase,patterson1944ambiguities}.

The phase retrieval problem is evidently ill-posed and notoriously challenging to solve. In recent years, many researchers have demonstrated that it can be reformulated as an optimization problem. Consequently, several algorithms have been proposed to address this problem, including PhaseLift \cite{candes2015phase}, PhaseCut \cite{waldspurger2015phase}, and PhaseMax \cite{goldstein2018phasemax}.

The PhaseLift algorithm is employed to address our discrete phase retrieval problem, which comprises two primary components: multiple structured illumination and lifting. Multiple structured illumination is designed to obtain additional measurements by utilizing masks, optical gratings, or oblique illuminations artificially. Lifting is intended to reformulate the problem as a semidefinite programming problem.

We employ masks $\{M_i\}_{i=1,\cdots,N_m}$ to implement the multiple structured illumination. Each mask, denoted as $M_i\in\mathbb R^{N\times N}$ for $i=1,\cdots,N_m$, is a diagonal matrix. Specifically, the first mask, $M_1=I$, is chosen as the identity matrix. The entries of the other masks are randomly set to 0 or 1 to create random diffraction patterns. By substituting the discrete source function
\[
\bm f:=(f(t_0),\cdots,f(t_N))^\top
\]
with sources using the masks, i.e.,
\[
^i\bm f:=M_i\bm f,\quad i=1,\cdots,N_m
\]
in the numerical scheme \eqref{FDM}, we obtain additional discrete solutions $\{ ^iu_0^n\}^{i=1,\cdots,N_m}_{n=0,\cdots,N}$ and thus more noisy data $\{^i\hat u_0^{n_\omega,\epsilon}\}_{n_\omega=1,\cdots,N_\omega}^{i=1,\cdots,N_m}$ in the frequency domain. These can be utilized to derive more phaseless Fourier modes $\{|^i\hat{f}^{\epsilon}(\omega_{n_\omega})|\}_{n_\omega=1,\cdots,N_\omega}^{i=1,\cdots,N_m}$. This procedure can be summarized as follows:
\[
\begin{CD}
\text{\fbox{$^i\bm f$}} @>{\phantom{\text{me}}}\eqref{FDM}{\phantom{\text{me}}}>\text{FDM}> \text{\fbox{$^iu_0^n$}}@>{\phantom{\text{m}}}\eqref{adnoisy}{\phantom{\text{m}}}>\text{noisy solution}> \text{\fbox{$^iu_0^{n,\epsilon}$}}@>{\phantom{\text{m}}}\text{discrete Fourier}{\phantom{\text{m}}}>\text{transform}> \text{\fbox{$^i\hat u_0^{n_{\omega},\epsilon}$}}@>{\phantom{\text{m}}}\eqref{hatFnoisy}{\phantom{\text{m}}}>\text{Spectral\, cutoff}> \text{\fbox{$|^i\hat{f}^{\epsilon}(\omega_{n_\omega})|$}}\,.
\end{CD}
\]

We refer to \cite{candes2015phase, gong2021numerical} for further details on the implementation of the PhaseLift method. Additionally, we suggest consulting \cite{chandra2019international} for access to the specific code used in the PhaseLift algorithm.

\subsubsection{Numerical examples}

In this subsection, we present three illustrative examples to demonstrate the effectiveness of the numerical approach. In all these numerical examples, the values for the final time $T$, as well as the numbers of subintervals in time $N$ and space $M$, are set as follows:
\[
T=4\pi,~N=100,~M=128.
\]		
Furthermore, to approximate the variance of the solution in \eqref{hatFnoisy}, we take a total of $P$ sample paths. The specific choice for the parameter $P$ will be detailed in each individual numerical example.

\begin{example}\label{exm1}
Consider a smooth function $f(t) = \sin(t)\exp (-t/6)$.
\end{example}

In Example \ref{exm1}, multiple tests are conducted to illustrate the impact of various parameters on the numerical implementation. These parameters include the regularization parameter $W$, the quantity of masks $N_m$, the number of sample paths $P$, the Hurst parameter $H$, the order $\bm{\alpha}$ of the fractional derivative, and the noise level $\epsilon$.

\begin{figure}[h]
  \centering
  \includegraphics[width=0.3\textwidth]{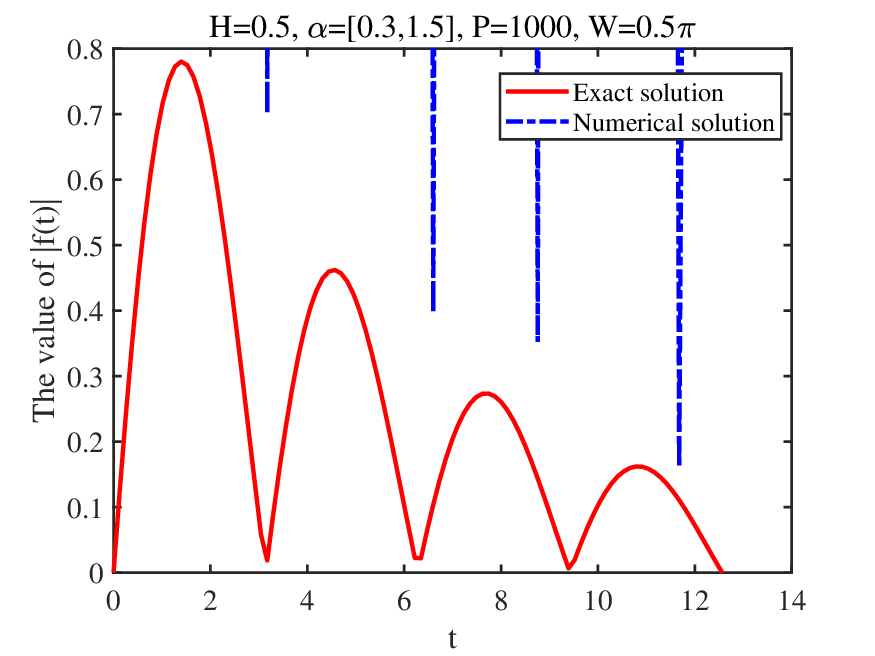}
  \includegraphics[width=0.3\textwidth]{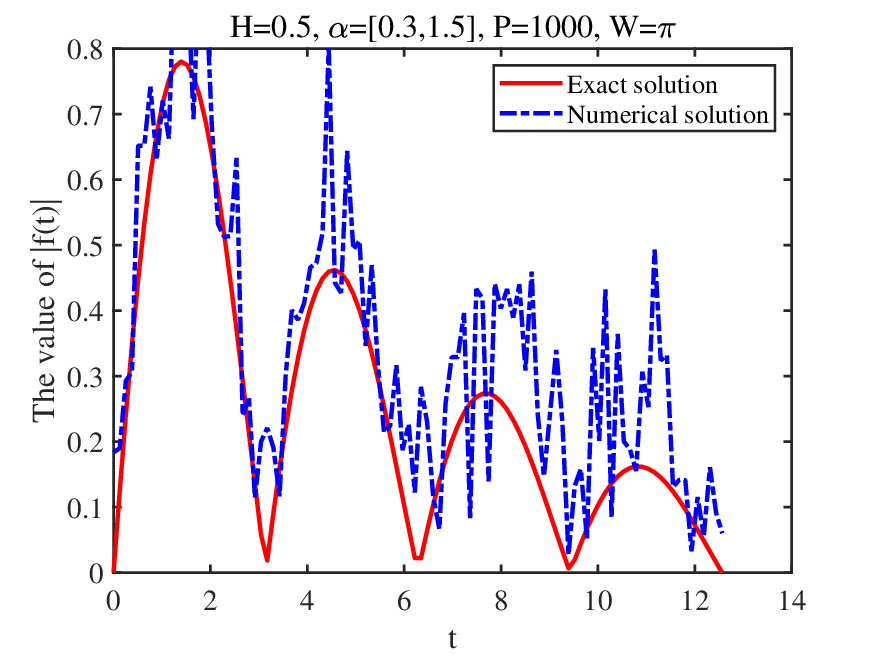}
  \includegraphics[width=0.3\textwidth]{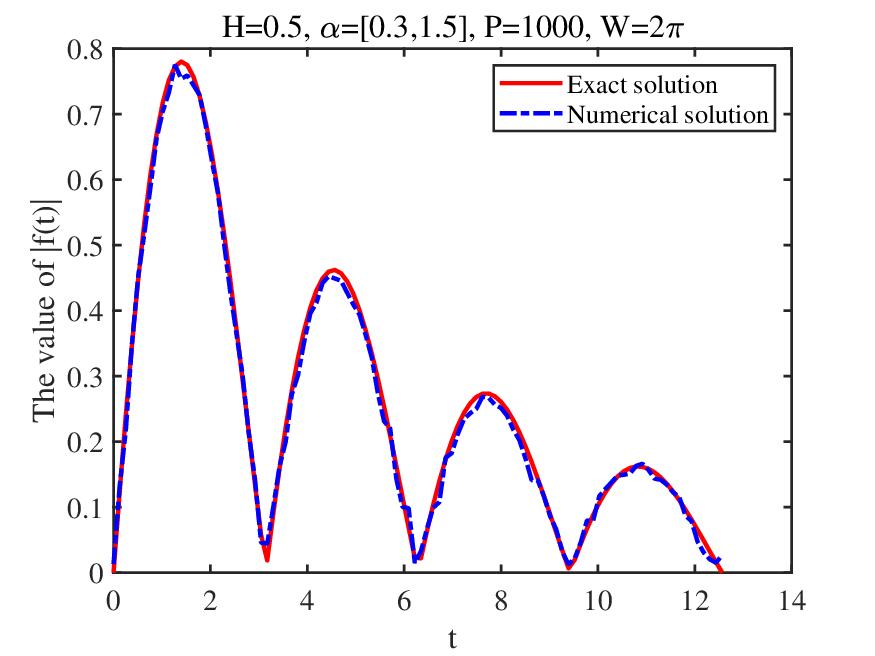}\\
  \includegraphics[width=0.3\textwidth]{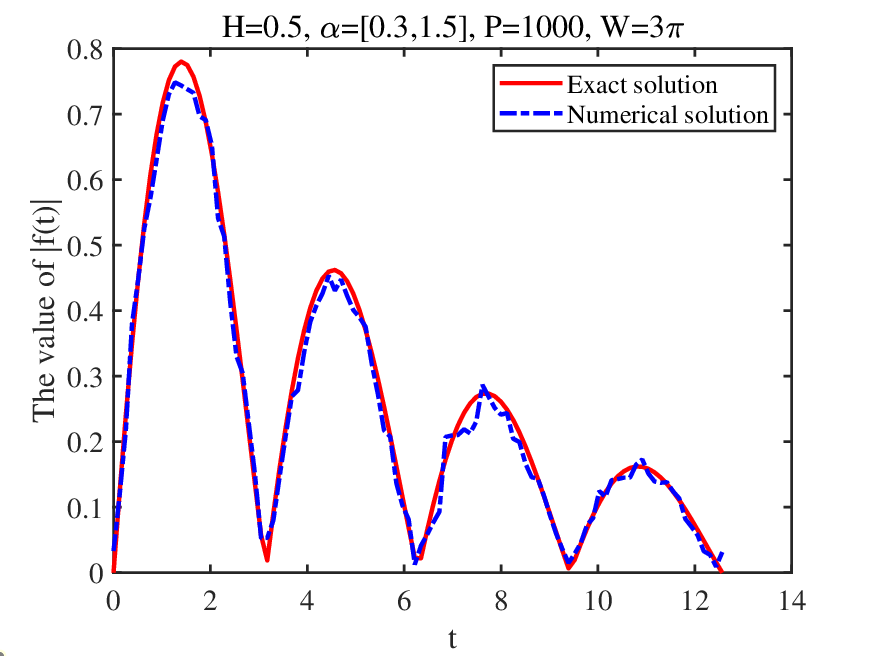}
  \includegraphics[width=0.3\textwidth]{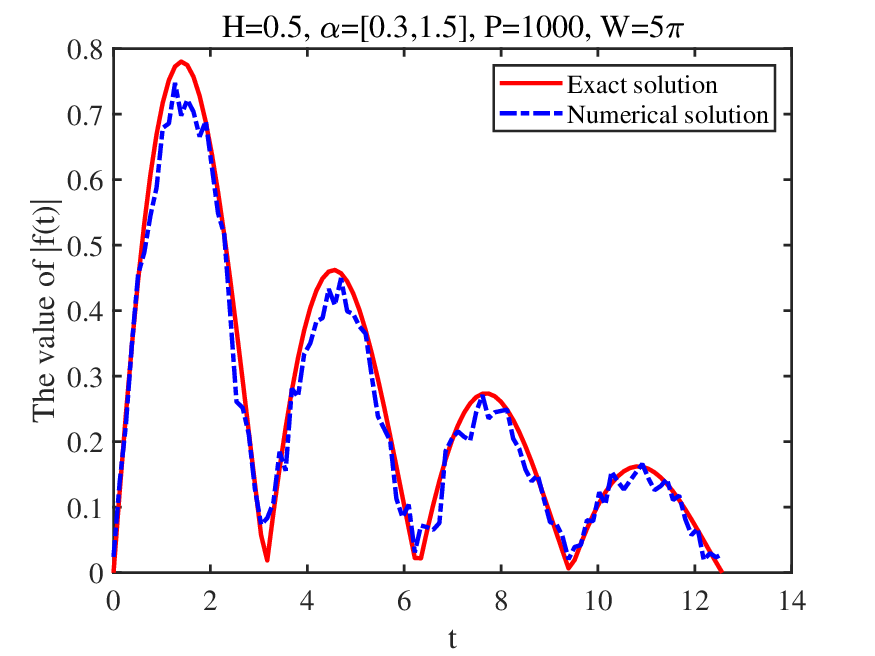}
  \includegraphics[width=0.3\textwidth]{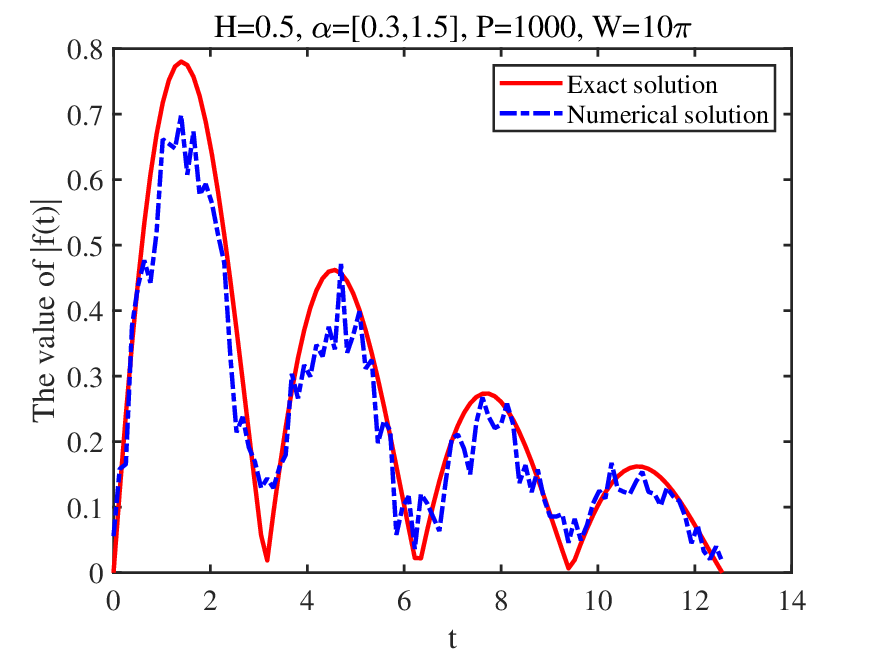}
  \caption{Example \ref{exm1}: Reconstruction of $|f(t)|$ with varying values of $W=0.5\pi$, $\pi$, $2\pi$, $3\pi$, $5\pi$, $10\pi$, while keeping other parameters fixed ($H=0.5,~\bm{\alpha}=[0.3,1.5],~P=1000,~N_m=60,~\epsilon=5\%$).}\label{e1H5diffW}
\end{figure}

Figure \ref{e1H5diffW} presents the numerical results for the reconstruction of $|f(t)|$ in Example \ref{exm1}, utilizing different spectral cut-off regularization parameters, specifically $W=0.5\pi, \pi, 2\pi, 3\pi, 5\pi, 10\pi$. The remaining parameters are held constant: $H=0.5$, $\bm{\alpha}=[0.3,1.5]$, $P=1000$, $N_m=60$, and $\epsilon=5$. The results demonstrate that the reconstruction quality is not satisfactory when $W$ is excessively small, indicating insufficient information acquisition in the frequency domain, or when it is excessively large, leading to instability in the inverse problem. Consequently, it is crucial to select appropriate regularization parameters, tailored to the specific case at hand. To guide this selection, we refer to the values of $R(\omega)$ in relation to $\omega$, as illustrated in Figure \ref{denH357}. In forthcoming numerical tests, we adopt $W=10\pi$ for cases with $\bm{\alpha}=[0.2,0.3]$ and select $W=3\pi$ for those with $\bm{\alpha}=[0.3,1.5]$ or $[1.1,1.5]$.

\begin{figure}[h]
  \centering
  \includegraphics[width=0.3\textwidth]{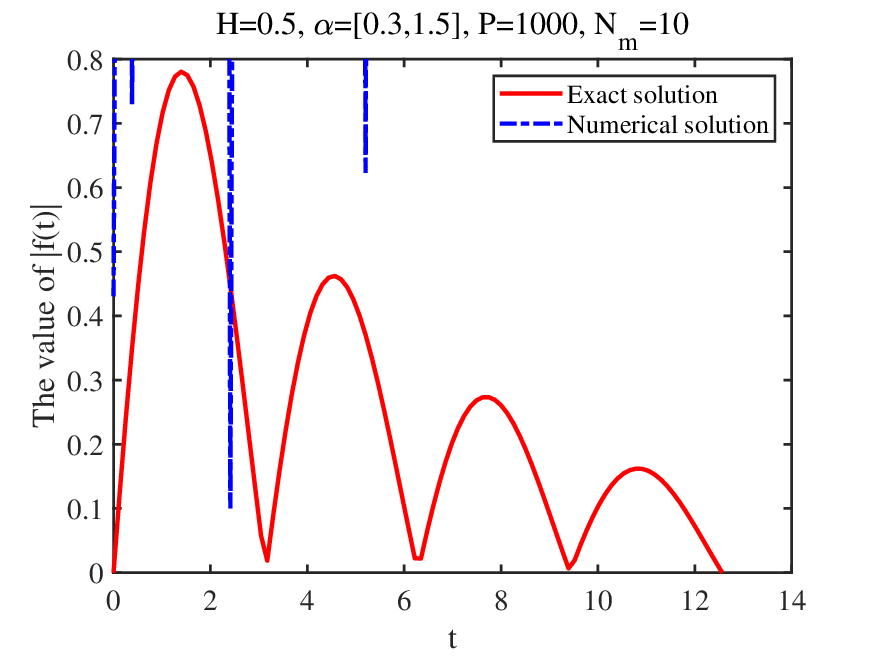}
  \includegraphics[width=0.3\textwidth]{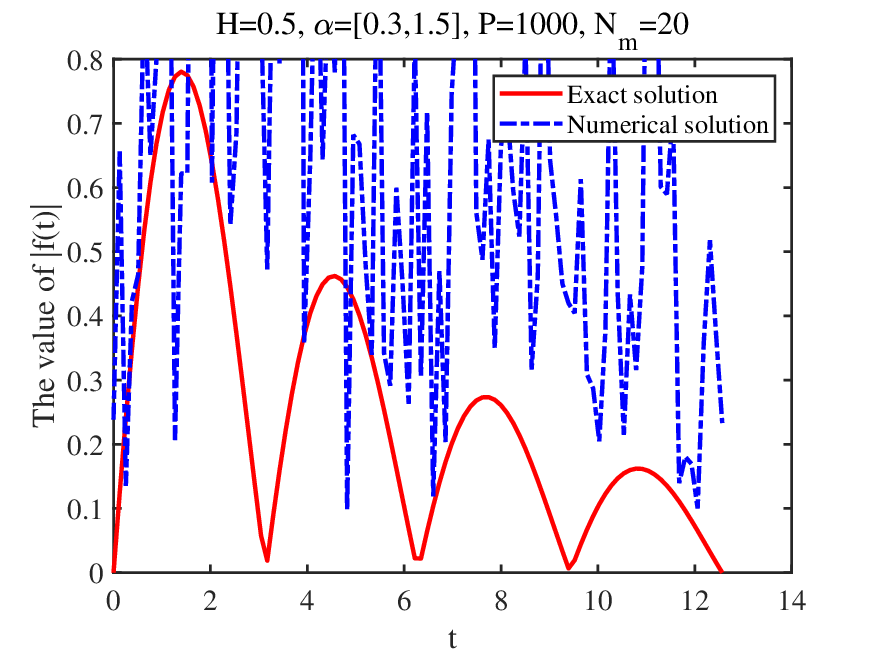}
  \includegraphics[width=0.3\textwidth]{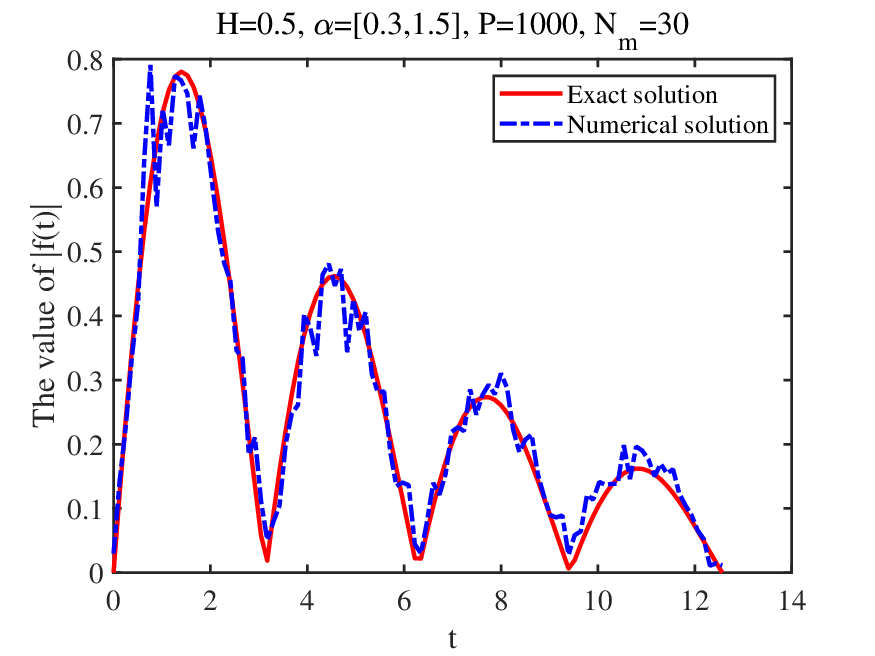}\\
  \includegraphics[width=0.3\textwidth]{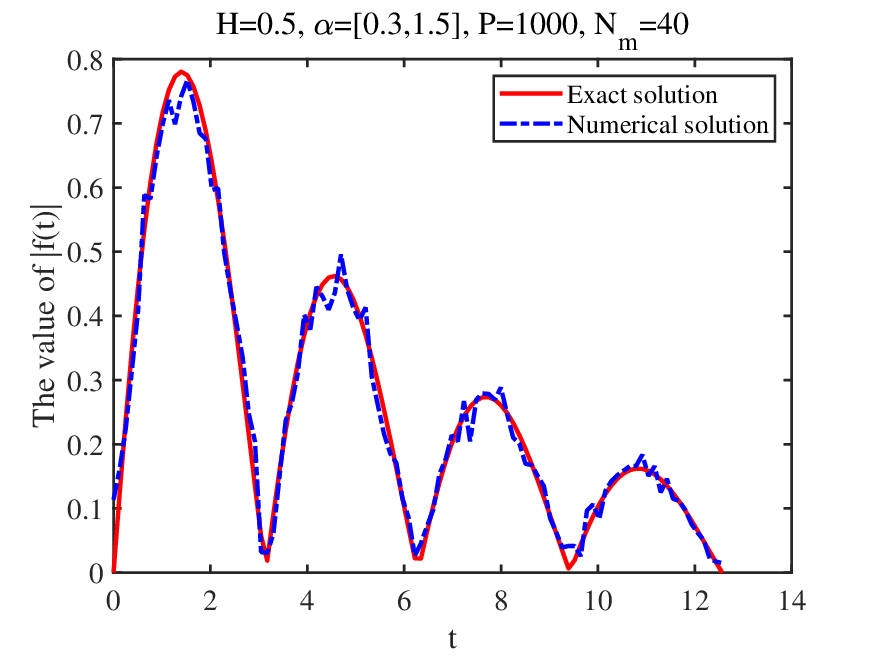}
  \includegraphics[width=0.3\textwidth]{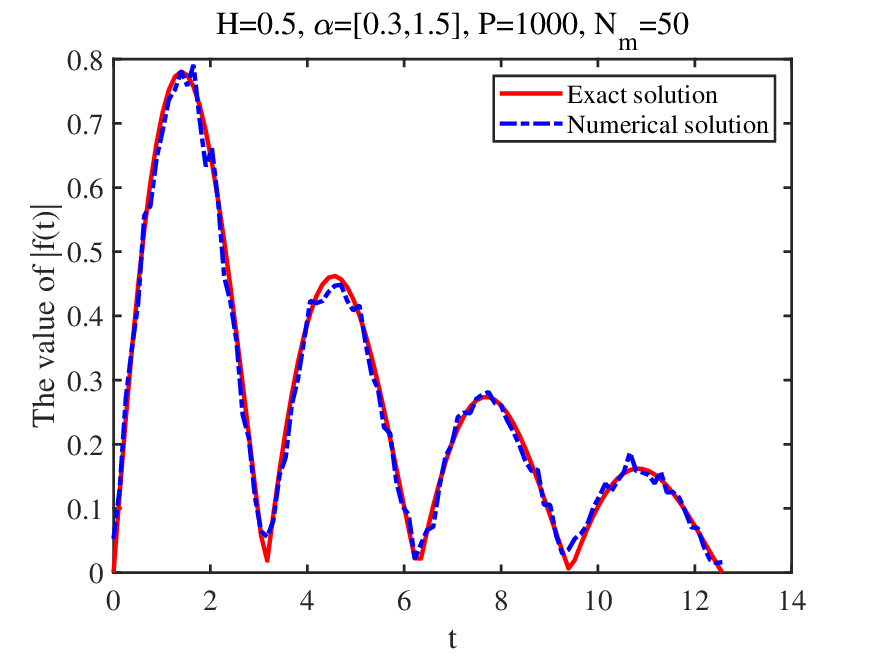}
  \includegraphics[width=0.3\textwidth]{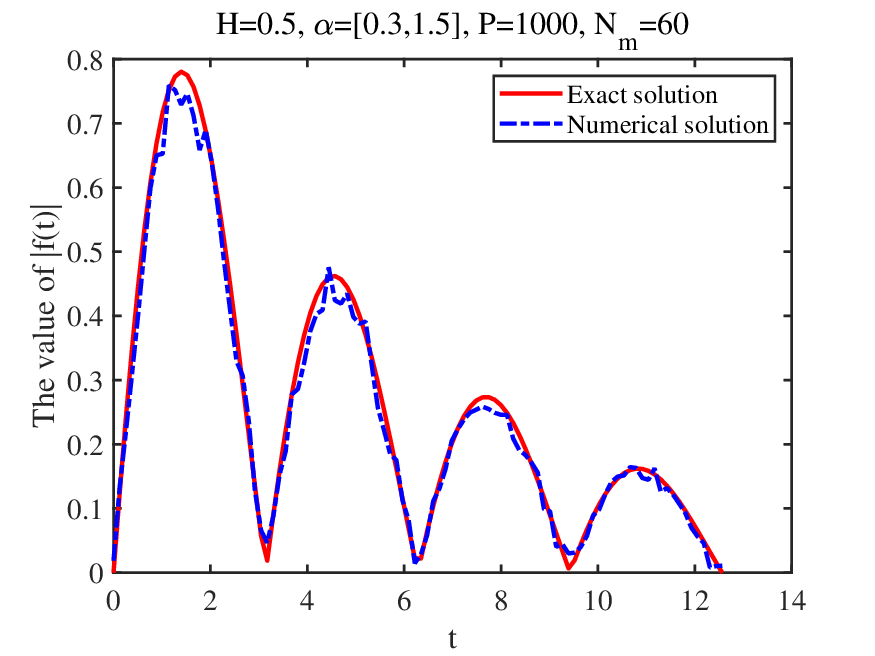}
  \caption{Example \ref{exm1}: Reconstruction of $|f(t)|$ with varying values of $N_m=10:10:60$, while keeping other parameters fixed ($H=0.5$, $\bm{\alpha}=[0.3,1.5]$, $P=1000$, $W=3\pi$, $\epsilon=5\%$).}\label{e1H5diffN}
\end{figure}

In Figure \ref{e1H5diffN}, we investigate the impact of varying the number of masks, denoted as $N_m$, in Example \ref{exm1}, while maintaining constant values for $W=3\pi$, $H=0.5$, $\bm{\alpha}=[0.3,1.5]$, $P=1000$, and $\epsilon=5$. The findings illustrate the necessity of employing a sufficient number of masks to ensure the acquisition of an adequate quantity of diffraction patterns, thereby enabling an accurate reconstruction. Based on the numerical results depicted in Figure \ref{e1H5diffN}, for subsequent numerical tests, we always choose $N_m=60$.

\begin{figure}[h]
  \centering
  \includegraphics[width=0.3\textwidth]{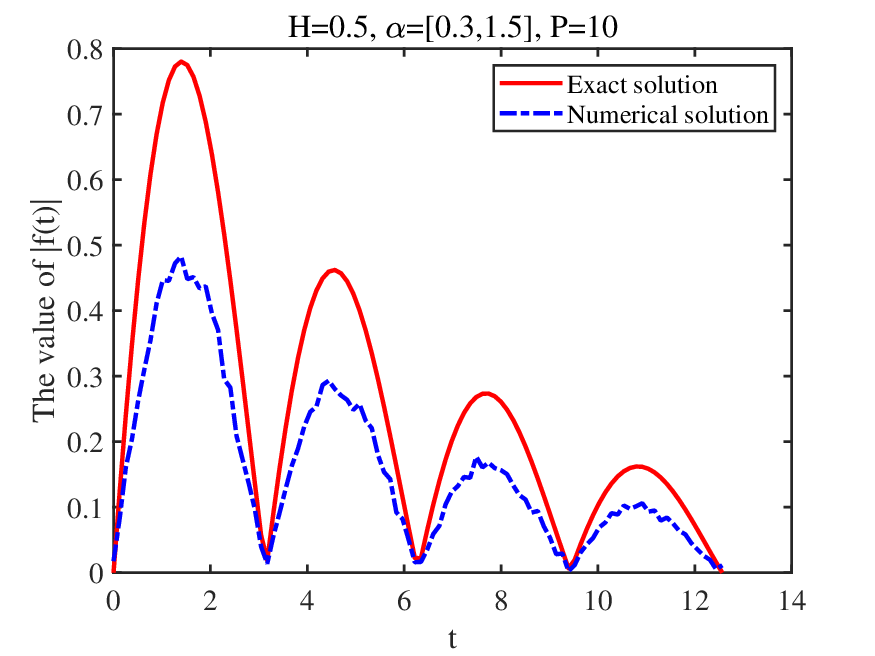}
  \includegraphics[width=0.3\textwidth]{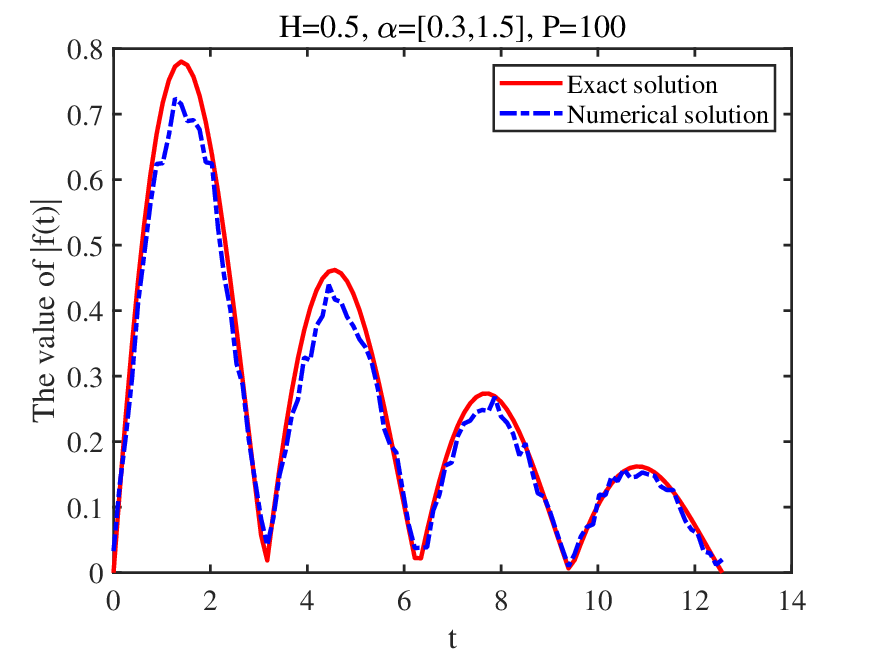}
  \includegraphics[width=0.3\textwidth]{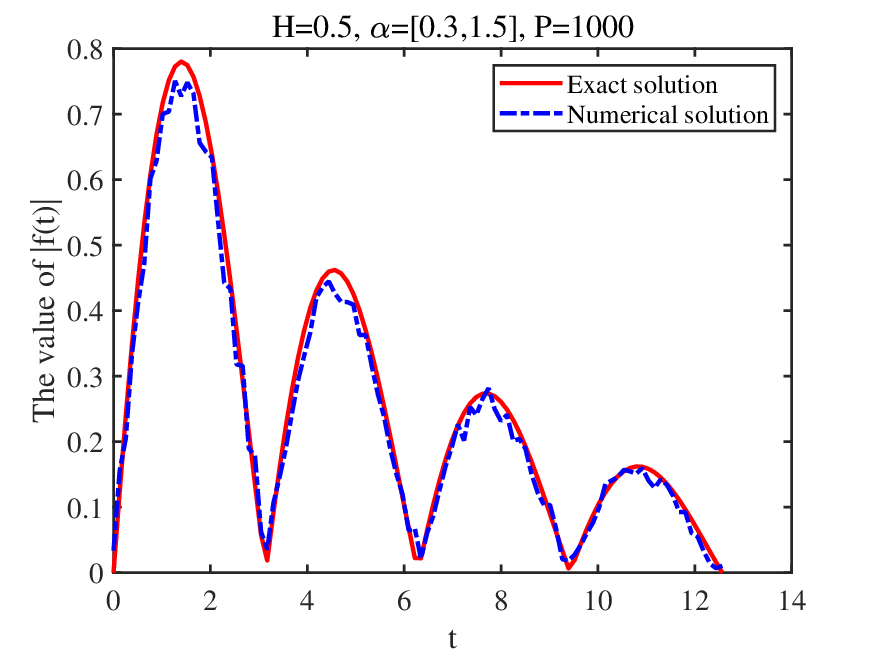}
  \caption{Example \ref{exm1}: Reconstruction of $|f(t)|$ with varying values of $P=10,100,1000$, while keeping other parameters fixed ($H=0.5$, $\bm{\alpha}=[0.3,1.5]$, $W=3\pi$, $N_m=3\pi$, $\epsilon=5\%$).}\label{e1H5diffP}
\end{figure}

Figure \ref{e1H5diffP} displays the numerical results of the reconstruction of $|f(t)|$ in Example \ref{exm1} under different sample path quantities, denoted as $P=10,100,1000$, while maintaining fixed values for $H=0.5$, $\bm{\alpha}=[0.3,1.5]$, $W=3\pi$, $N_m=3\pi$, and $\epsilon=5$. The observations indicate that the quality of reconstruction improves as more sample paths are employed to approximate the solution's variance, aligning with the principles of the law of large numbers. Notably, the numerical results suggest that a satisfactory level of reconstruction is already achieved with the choice of $P=1000$. Consequently, we use $P=1000$ as the fixed sample path quantity in subsequent experiments.

\begin{figure}[h]
  \centering
  \includegraphics[width=0.3\textwidth]{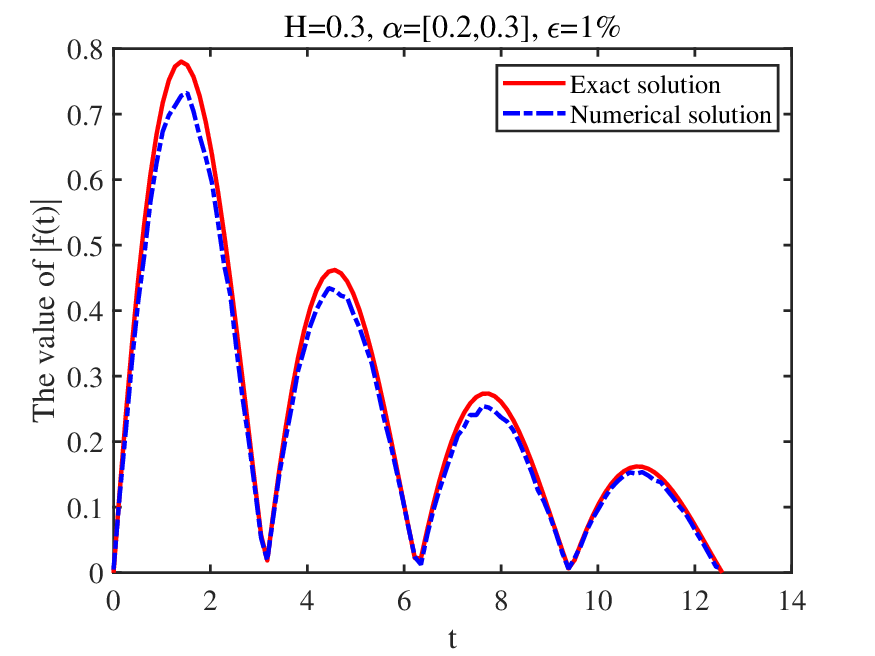}
  \includegraphics[width=0.3\textwidth]{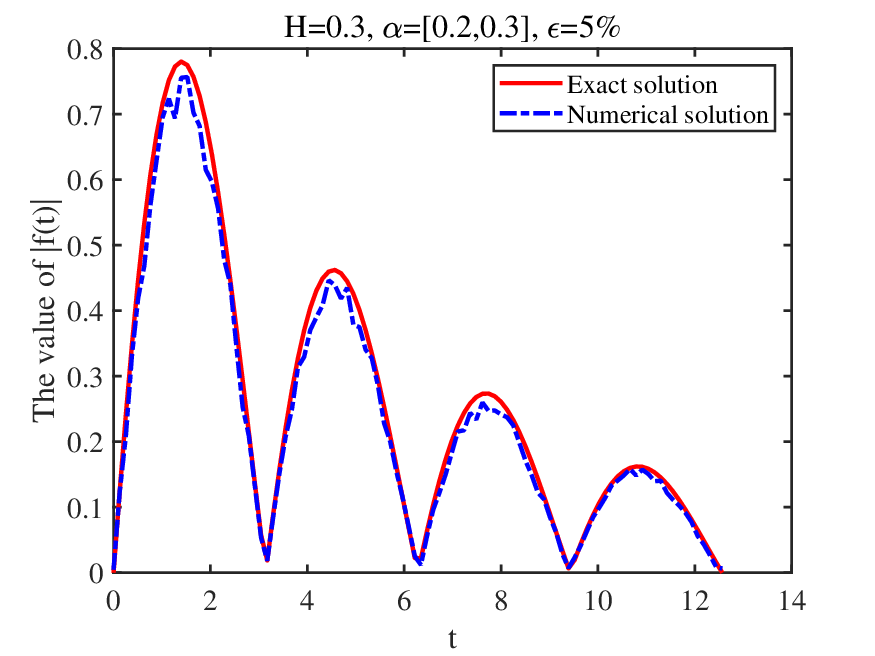}
  \includegraphics[width=0.3\textwidth]{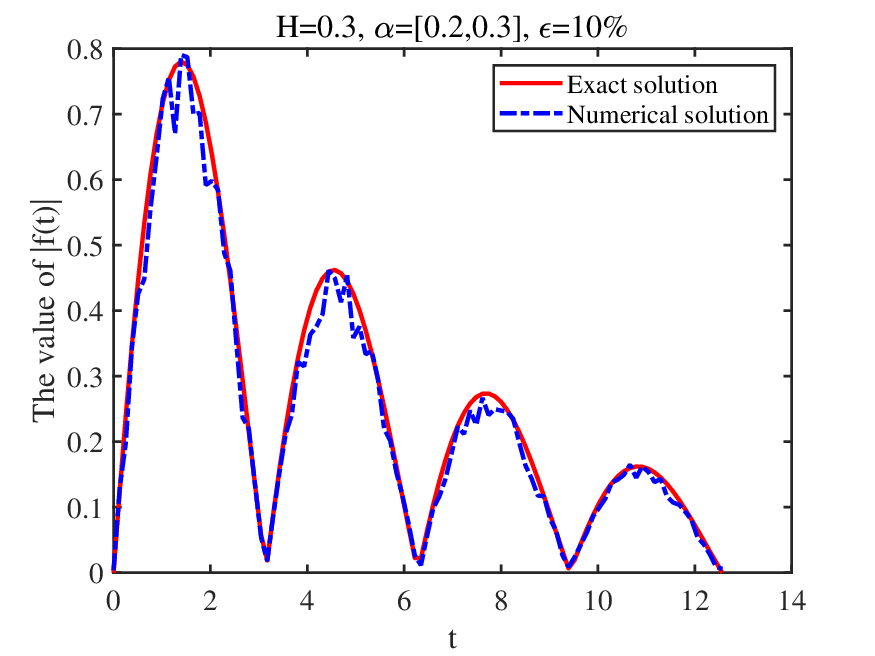}\\
  \includegraphics[width=0.3\textwidth]{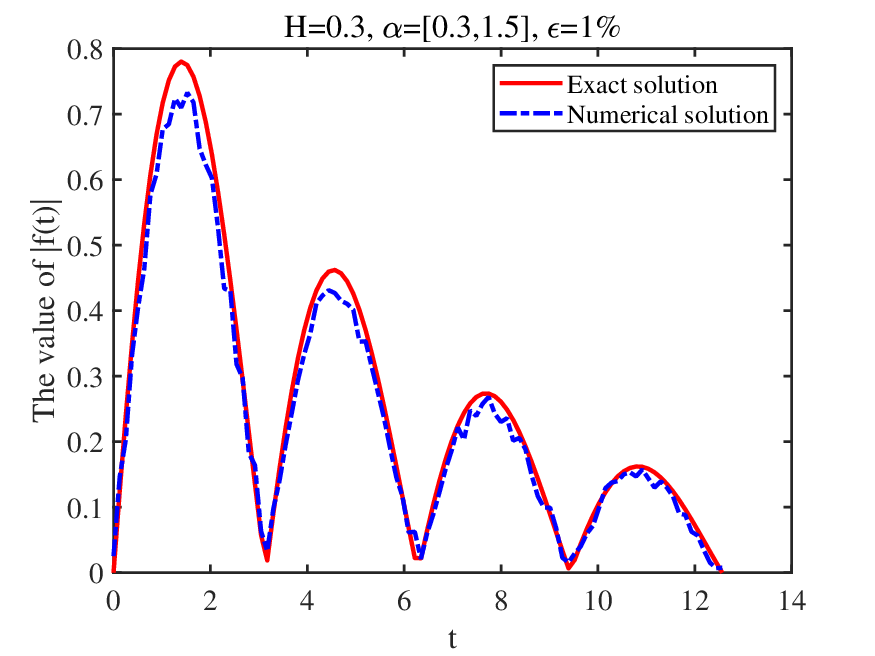}
  \includegraphics[width=0.3\textwidth]{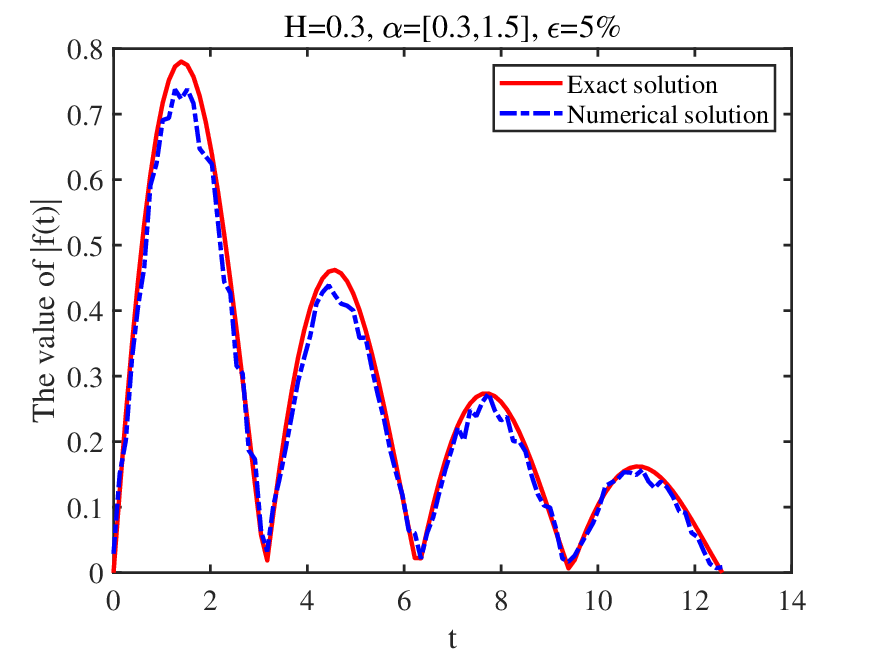}
  \includegraphics[width=0.3\textwidth]{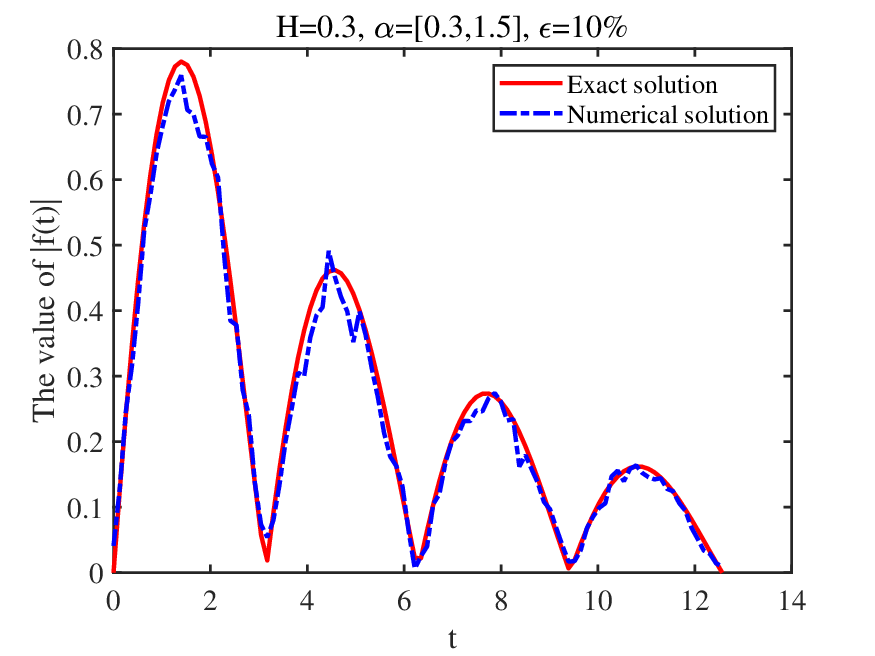}\\
  \includegraphics[width=0.3\textwidth]{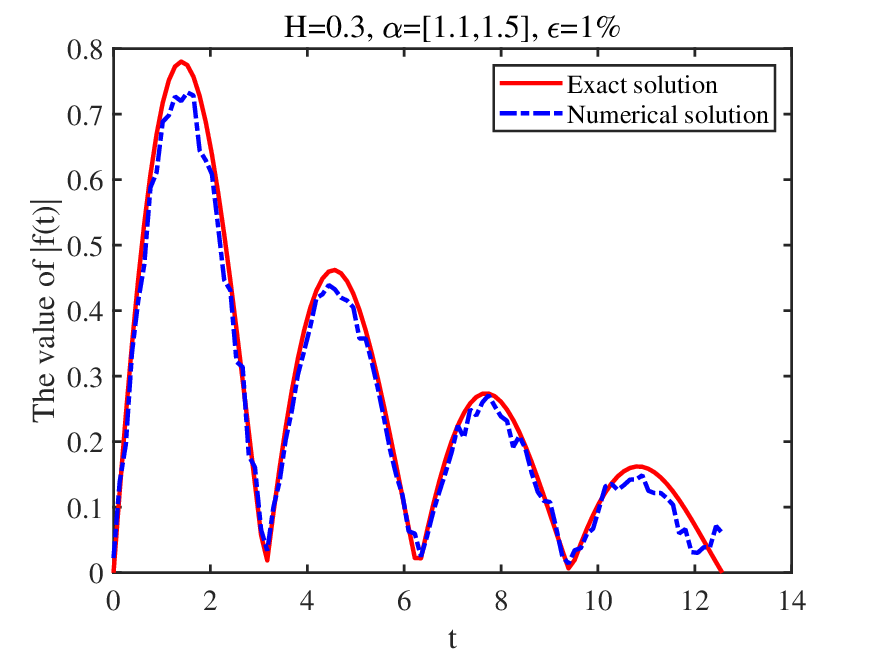}
  \includegraphics[width=0.3\textwidth]{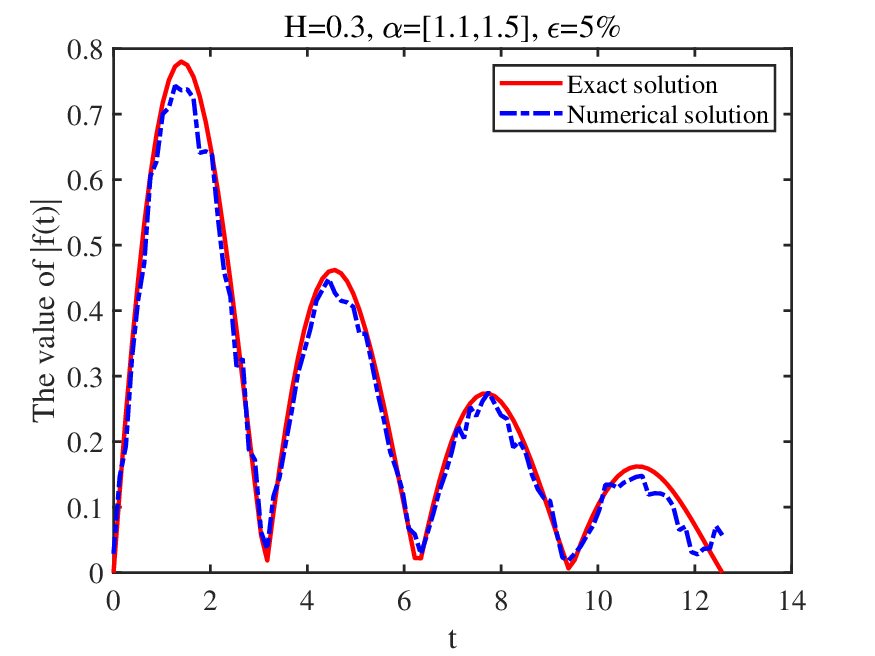}
  \includegraphics[width=0.3\textwidth]{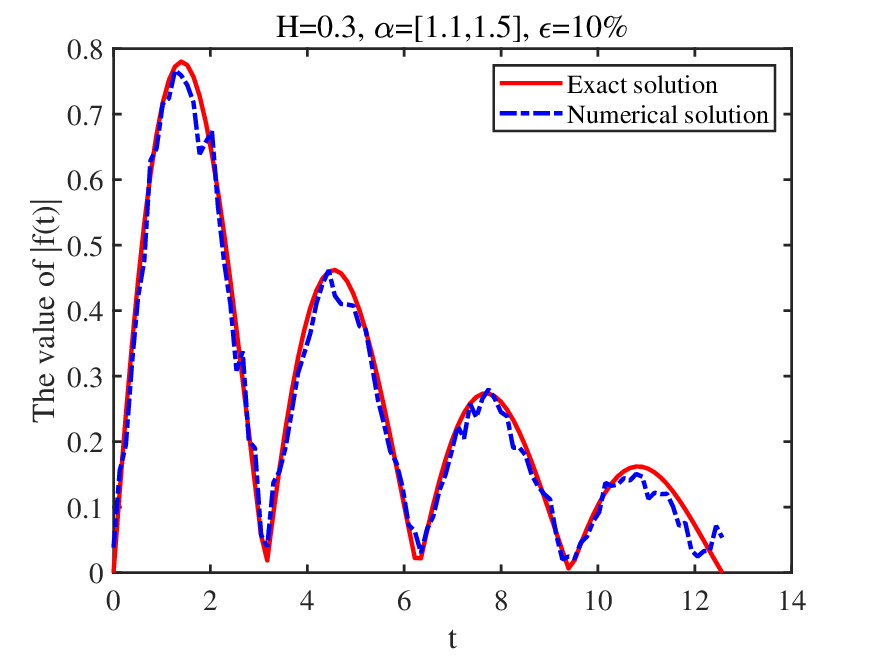}
  \caption{Example \ref{exm1}: Reconstruction of $|f(t)|$ with different levels of noise: (left) $\epsilon=1\%$, (middle) $\epsilon=5\%$, and (right) $\epsilon=10\%$, while varying the values of $\bm{\alpha}$ under the constant Hurst parameter $H=0.3$.}\label{e1H3}
\end{figure}

Figures \ref{e1H3}--\ref{e1H7} depict the influence of parameters $H=0.3,0.5,0.7$, $\bm{\alpha}=[0.2,0.3]$, $[0.3,1.5]$, $[1.1,1.5]$, and $\epsilon=1\%, 5\%, 10\%$, while keeping $W$, $N_m$, and $P$ fixed, as previously specified. Upon examination of the subfigures within each row, it becomes evident that, for fixed values of $H$ and $\bm{\alpha}$, the results exhibit relatively higher quality when the noise level, $\epsilon$, is reduced. Likewise, within each column of these figures, when both $H$ and $\epsilon$ are held constant, decreasing the value of $\bm{\alpha}$ leads to improved results. These trends align with the prior theoretical analysis. Furthermore, through a comparative assessment of the results at corresponding positions in Figures \ref{e1H3}--\ref{e1H7}, it appears that the outcomes are less sensitive to variations in the Hurst index $H$ when compared to alterations in other parameters.

\begin{figure}[h]
  \centering
  \includegraphics[width=0.3\textwidth]{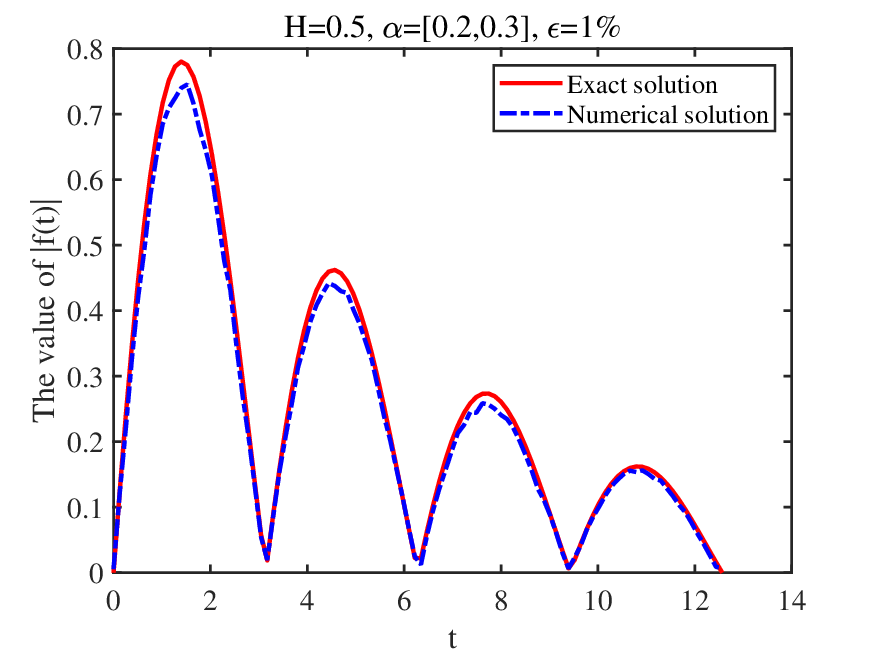}
  \includegraphics[width=0.3\textwidth]{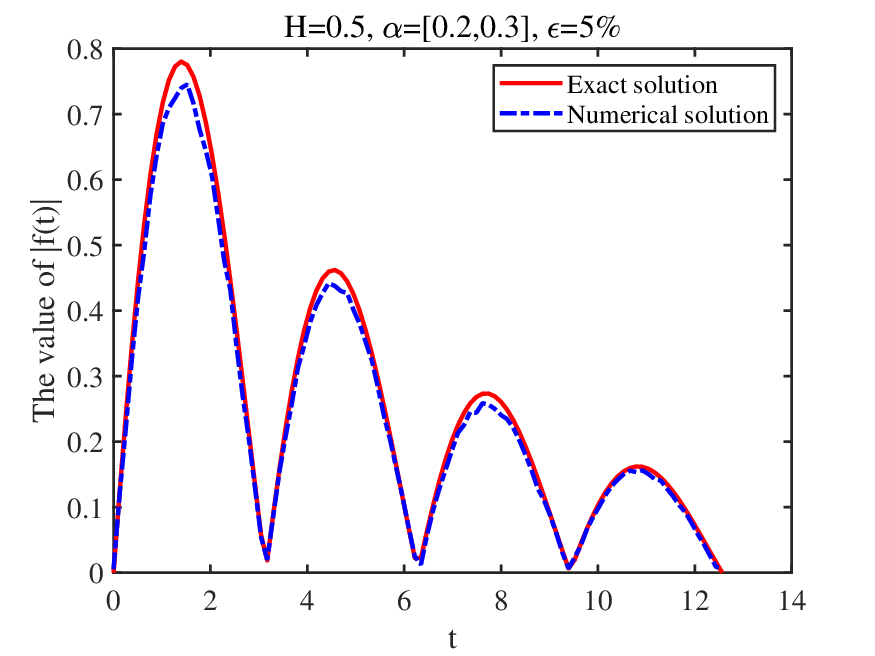}
  \includegraphics[width=0.3\textwidth]{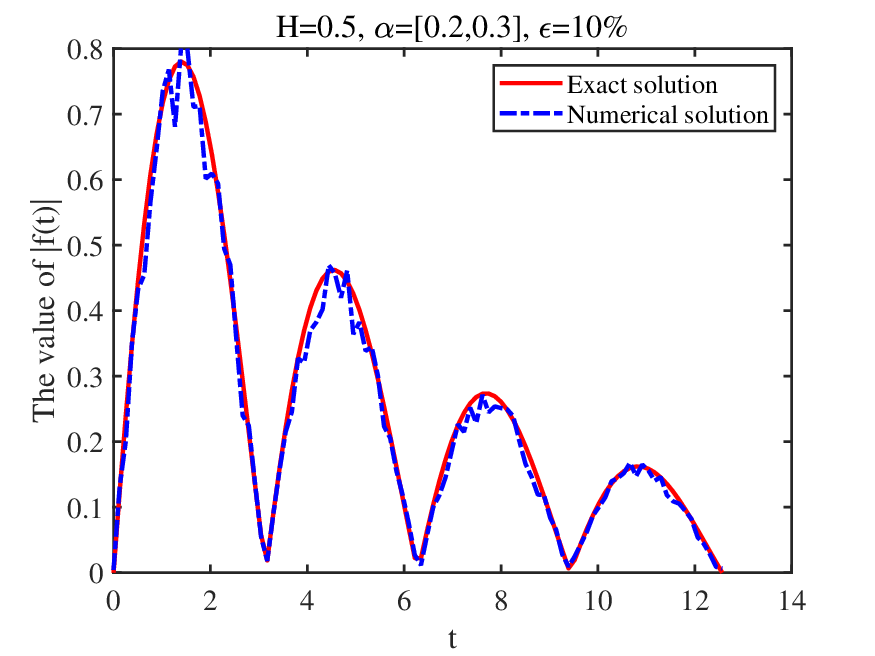}\\
  \includegraphics[width=0.3\textwidth]{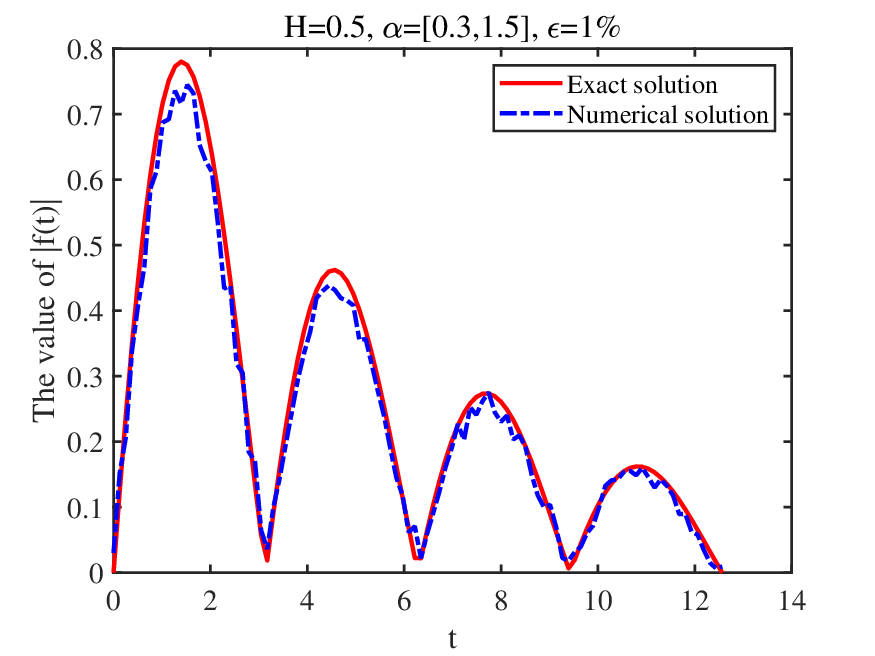}
  \includegraphics[width=0.3\textwidth]{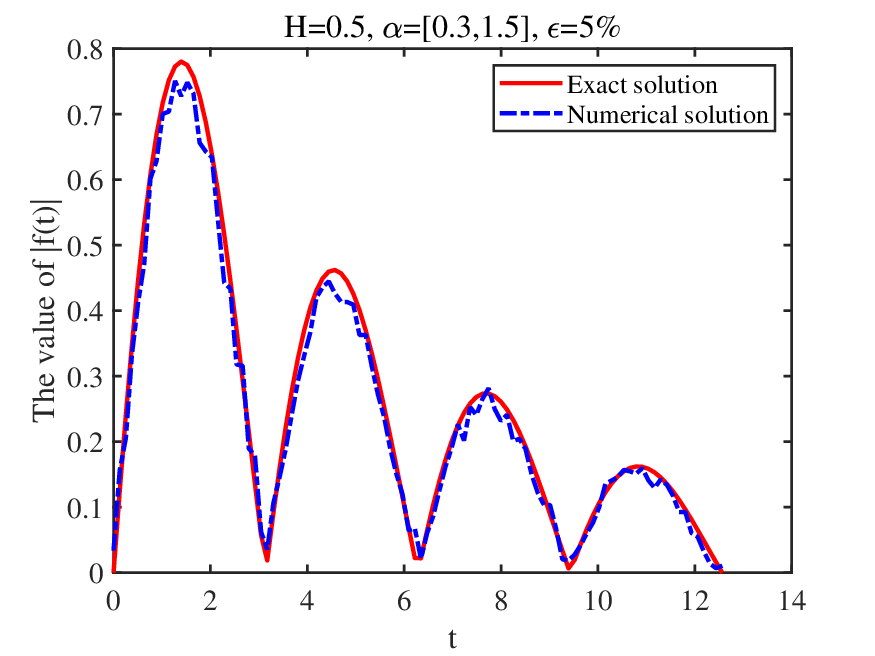}
  \includegraphics[width=0.3\textwidth]{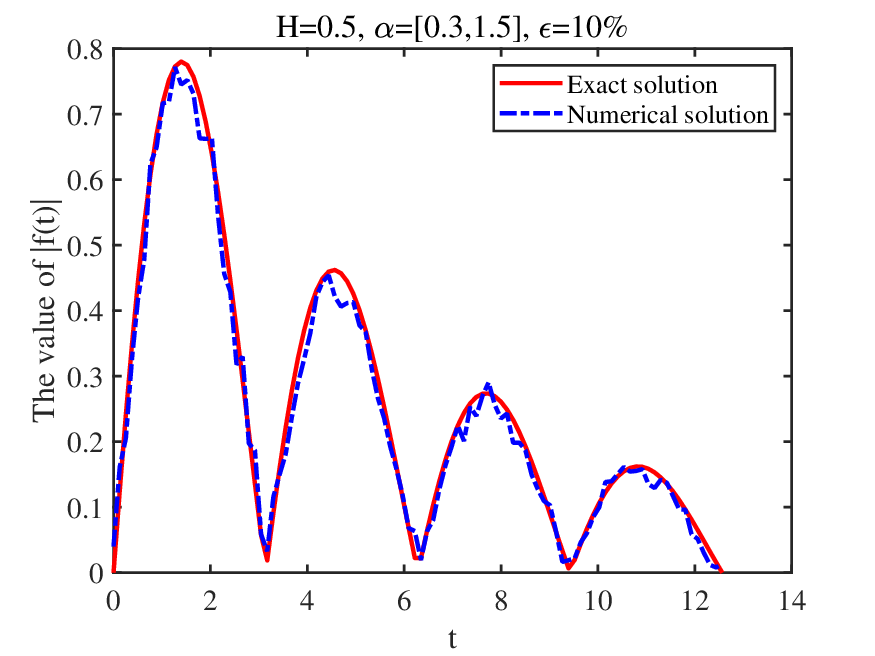}\\
  \includegraphics[width=0.3\textwidth]{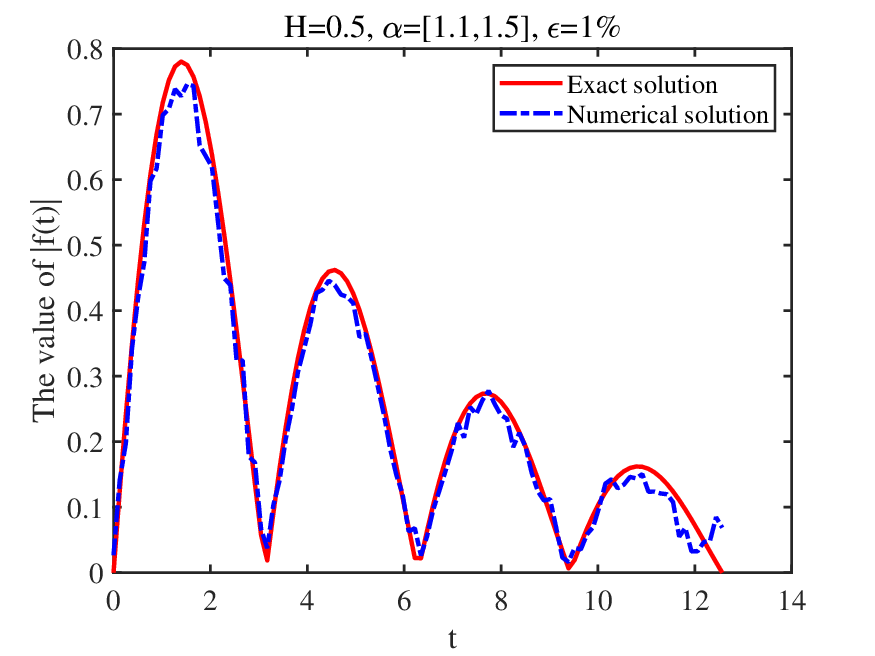}
  \includegraphics[width=0.3\textwidth]{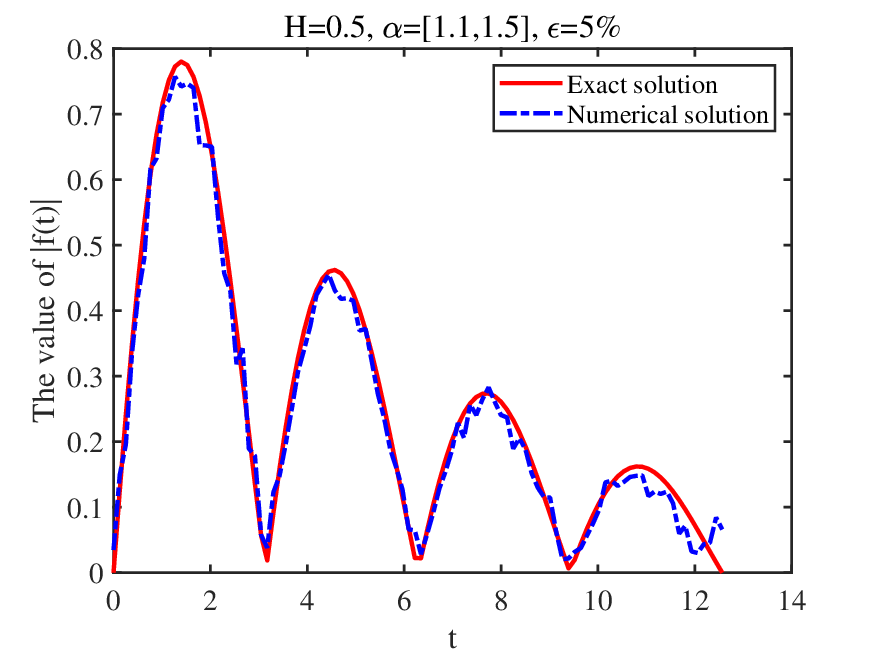}
  \includegraphics[width=0.3\textwidth]{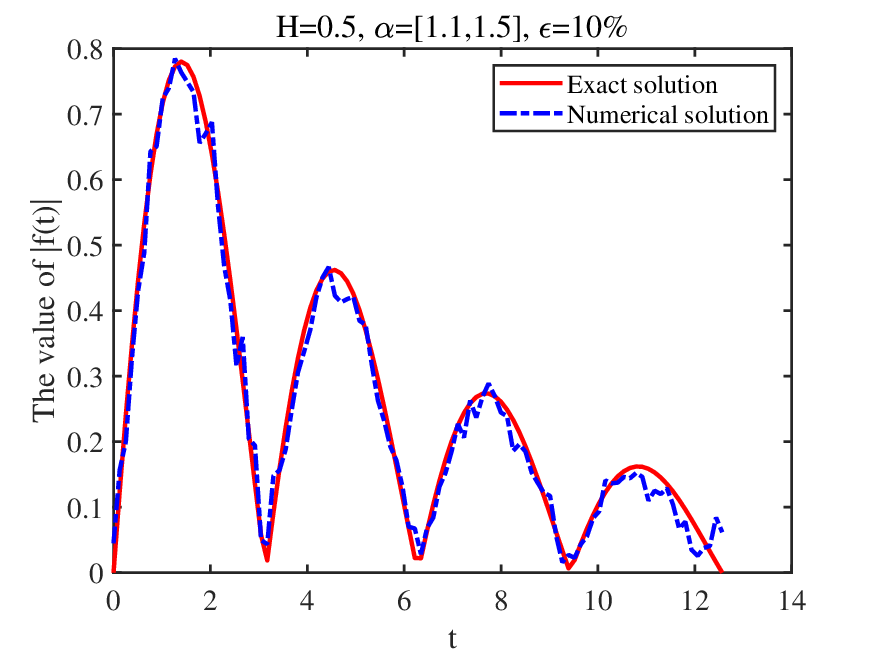}
  \caption{Example \ref{exm1}: Reconstruction of $|f(t)|$ with different levels of noise: (left) $\epsilon=1\%$, (middle) $\epsilon=5\%$, and (right) $\epsilon=10\%$, while varying the values $\bm \alpha$ under the constant Hurst parameter $H=0.5$.}\label{e1H5}
\end{figure}

\begin{figure}[h]
  \centering
  \includegraphics[width=0.3\textwidth]{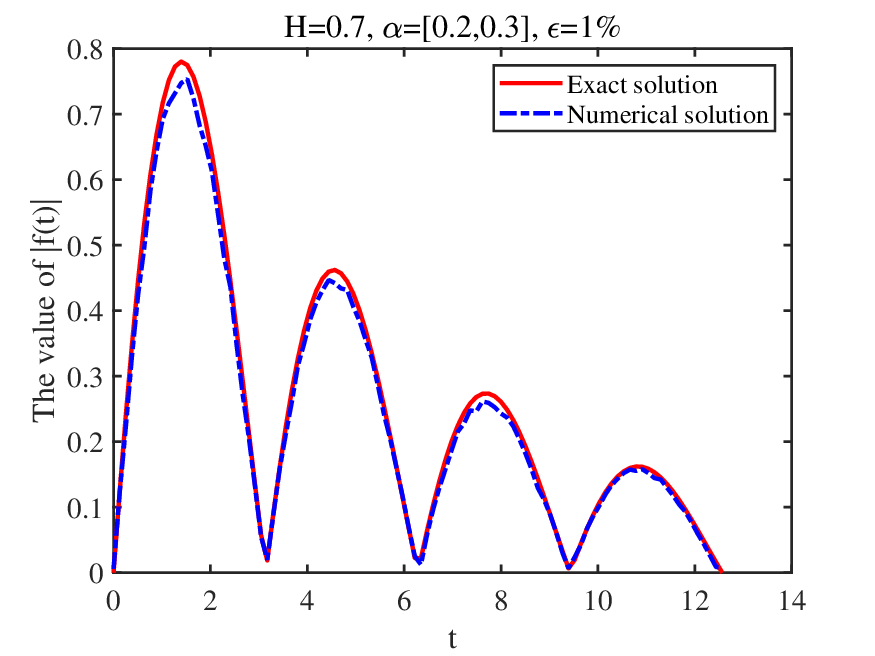}
  \includegraphics[width=0.3\textwidth]{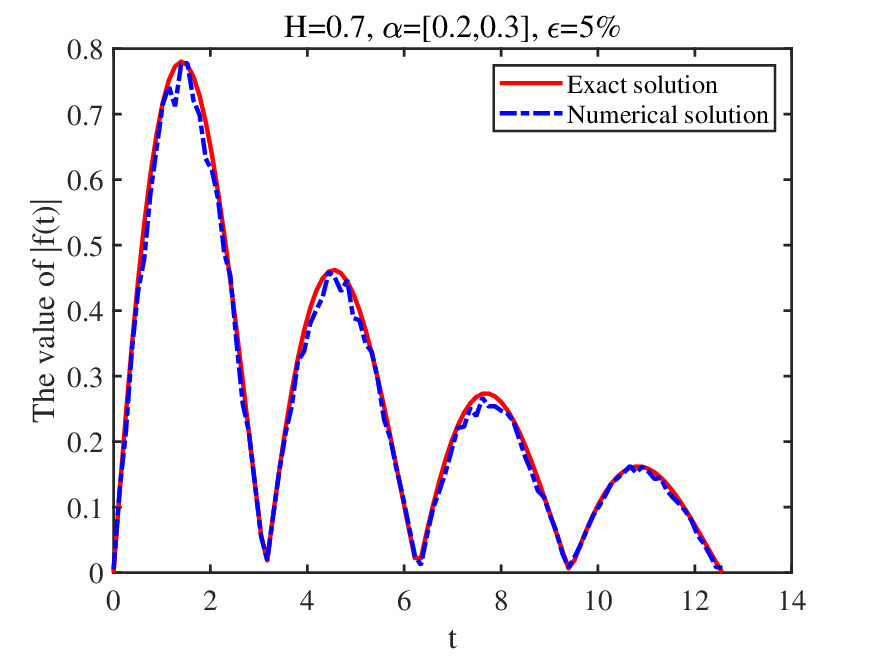}
  \includegraphics[width=0.3\textwidth]{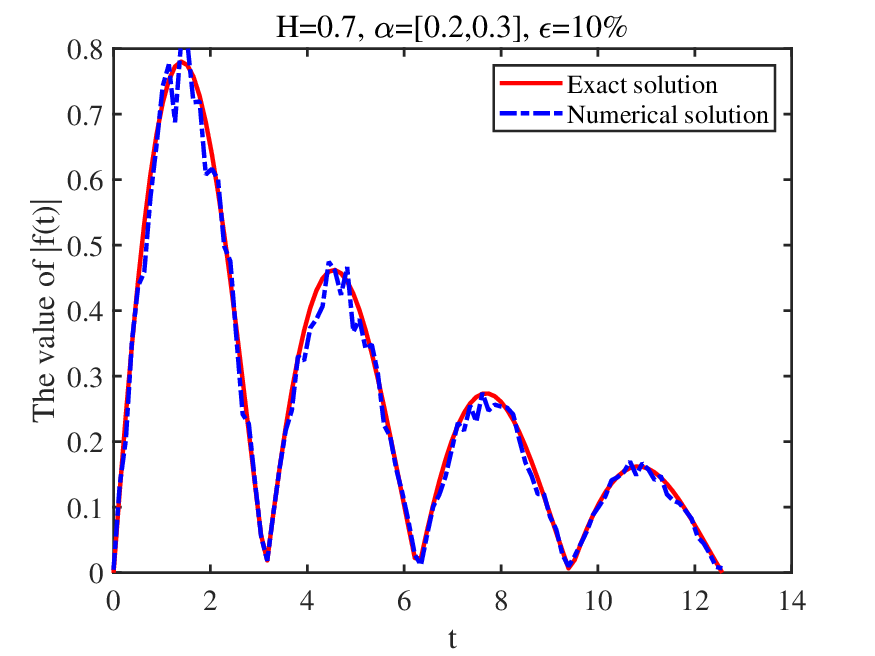}\\
  \includegraphics[width=0.3\textwidth]{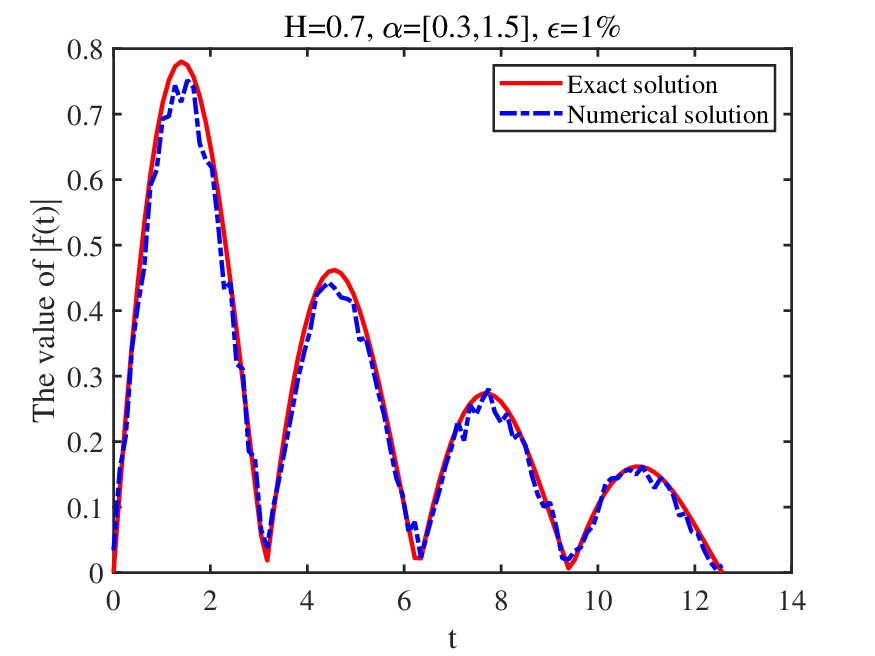}
  \includegraphics[width=0.3\textwidth]{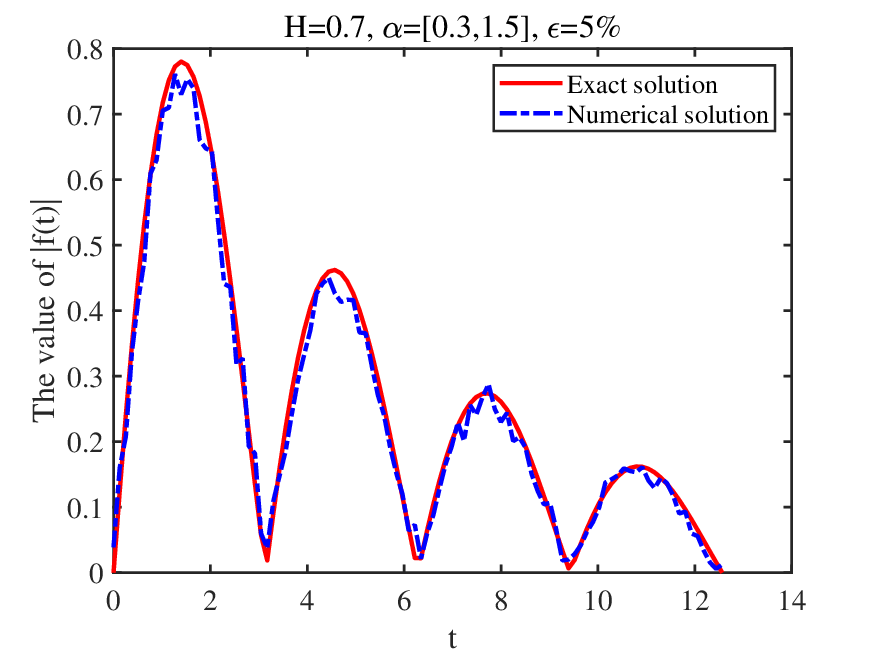}
  \includegraphics[width=0.3\textwidth]{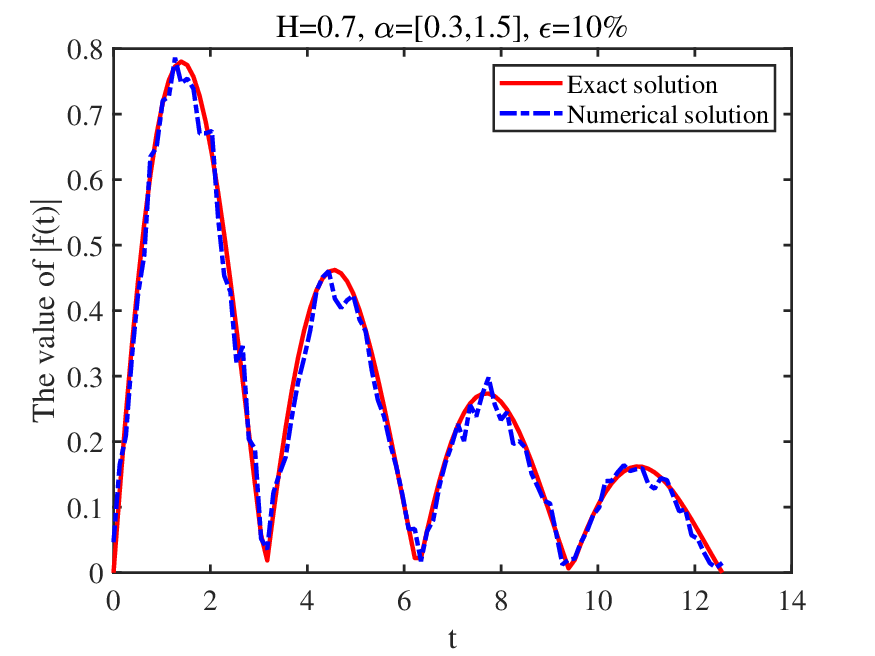}\\  \includegraphics[width=0.3\textwidth]{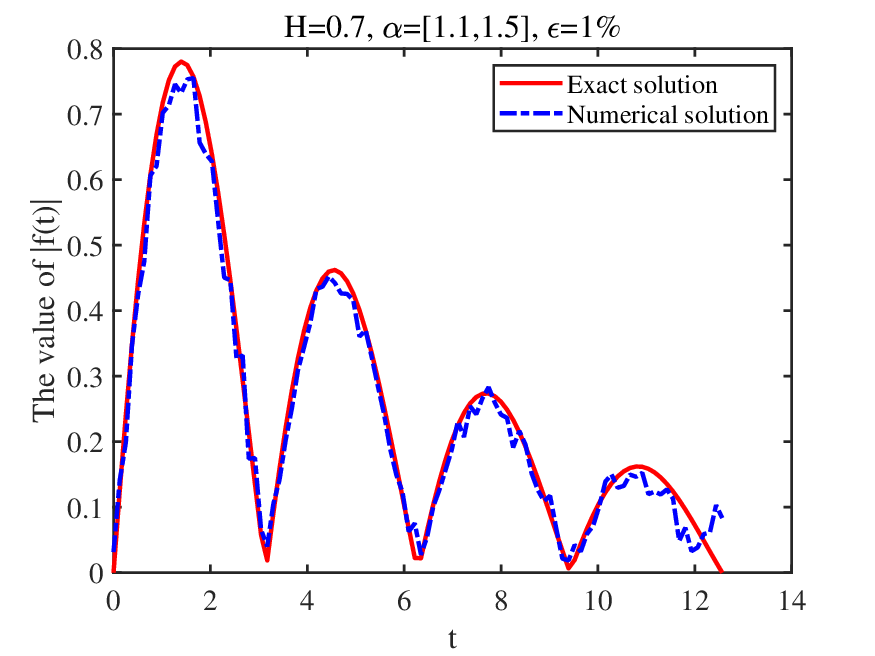}
  \includegraphics[width=0.3\textwidth]{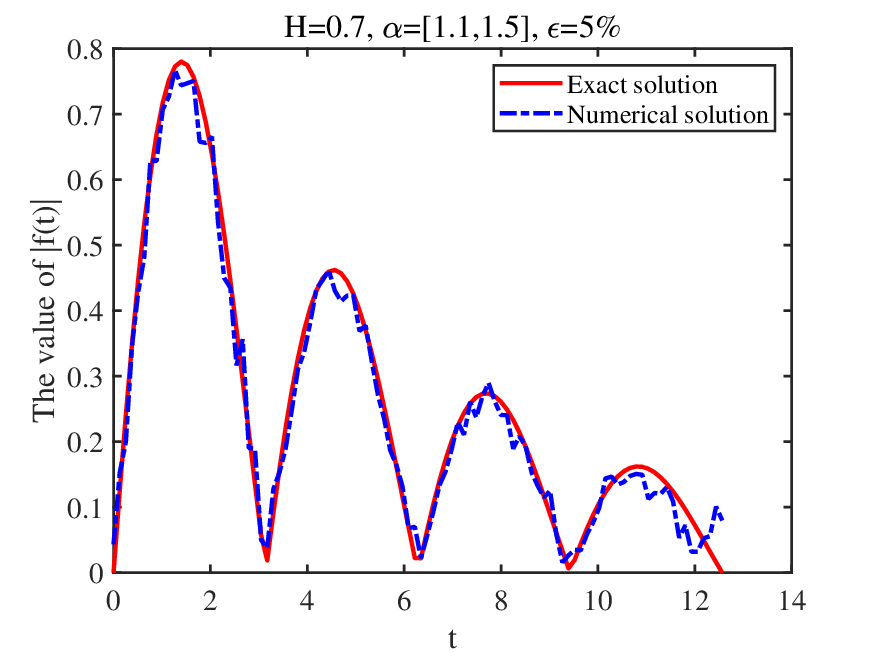}
  \includegraphics[width=0.3\textwidth]{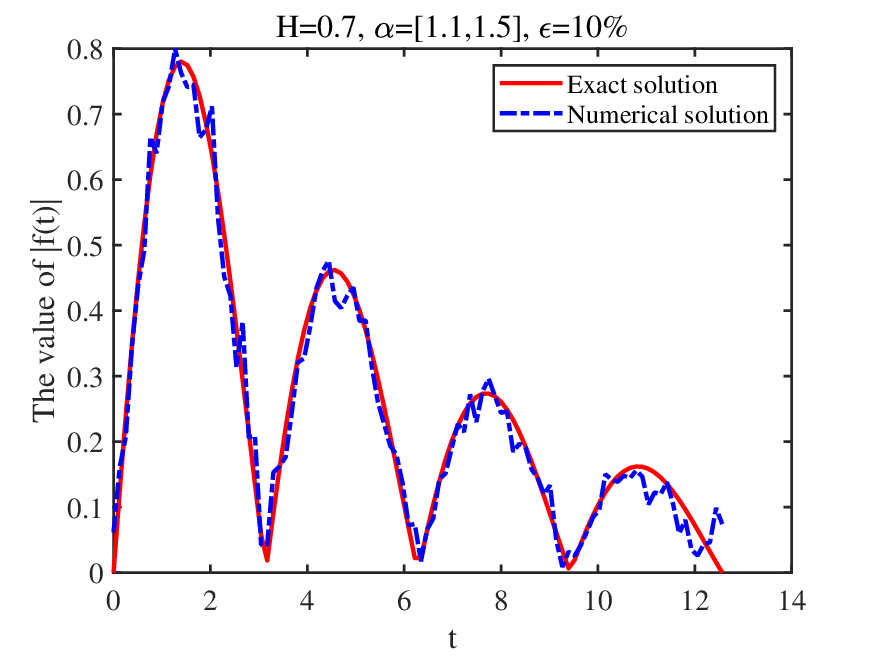}
  \caption{Example \ref{exm1}: Reconstruction of $|f(t)|$ with different levels of noise: (left) $\epsilon=1\%$, (middle) $\epsilon=5\%$, and (right) $\epsilon=10\%$, while varying the values  $\bm \alpha$ under the constant Hurst parameter $H=0.7$.}\label{e1H7}
\end{figure}

\begin{example}\label{exm2}
Consider a highly oscillatory function $f(t) = \sin(t)\cos(2t)$.
\end{example}

\begin{figure}[h]
  \centering
  \includegraphics[width=0.3\textwidth]{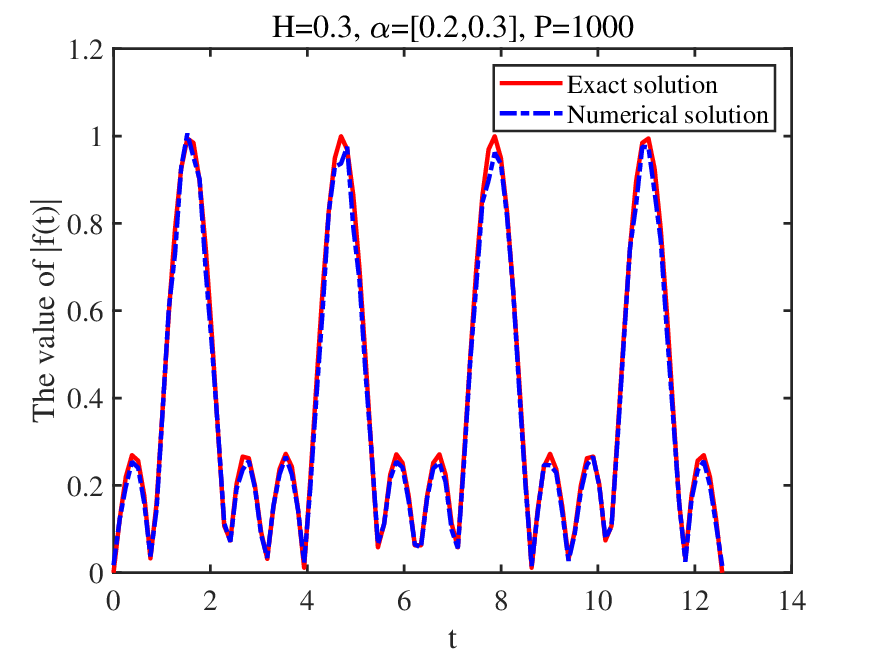}
  \includegraphics[width=0.3\textwidth]{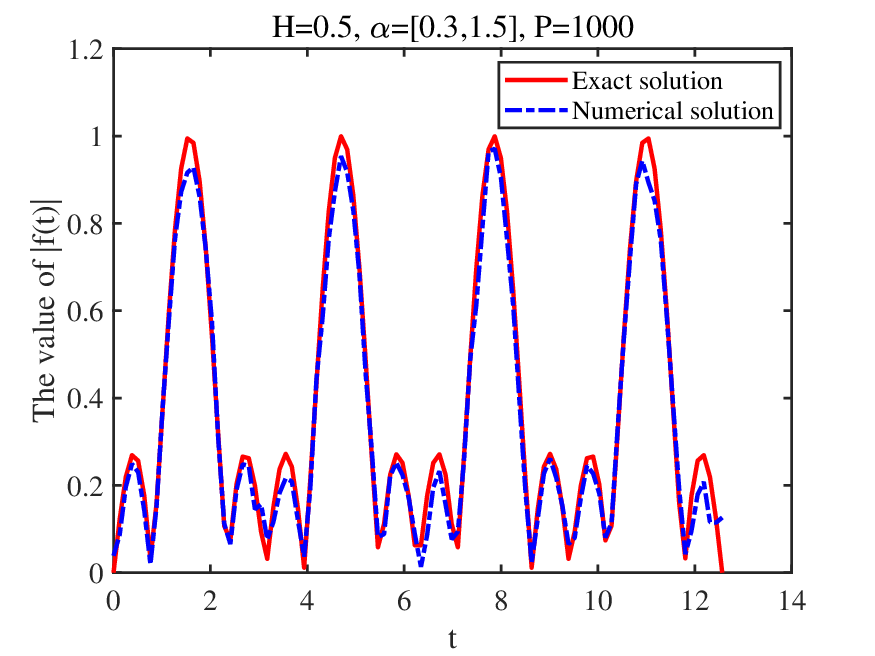}
  \includegraphics[width=0.3\textwidth]{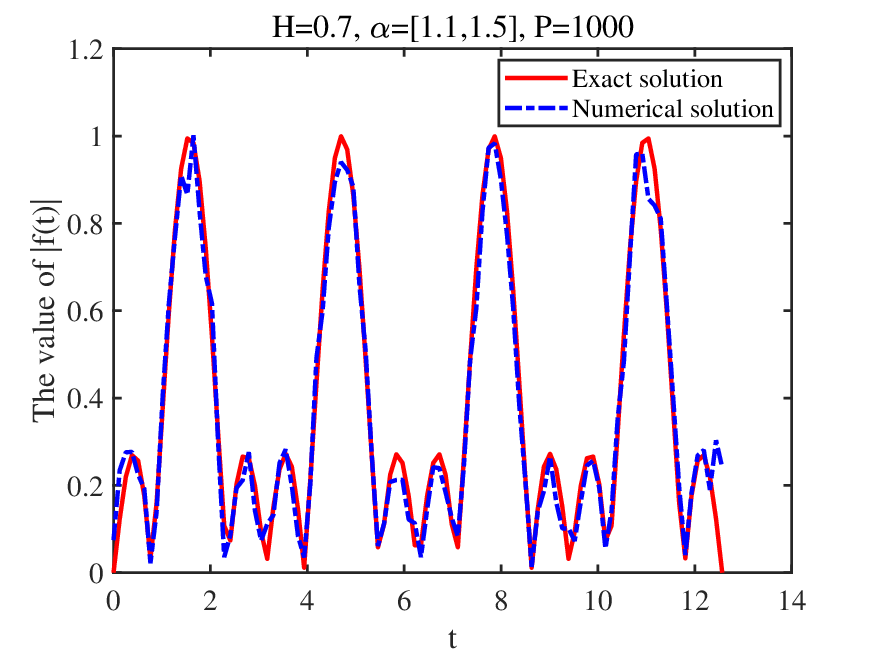}
  \caption{Example \ref{exm2}: Reconstruction of $|f(t)|$ with different combinations: (left) $H=0.3$ and $\bm{\alpha}=[0.2,0.3]$, (middle) $H=0.5$ and $\bm{\alpha}=[0.3,0.5]$, and (right) $H=0.7$ and $\bm{\alpha}=[1.1,1.5]$, while maintaining a consistent noise level of $\epsilon=5\%$.}\label{e2}
\end{figure}

In this example, we will not provide a detailed investigation of the impact of various parameters on the reconstruction, as the findings are similar to those in Example \ref{exm1}. Instead, we present the results using carefully chosen representative parameter values. Figure \ref{e2} shows the numerical results for Example \ref{exm2} across varying selections of $H$ and $\bm{\alpha}$ while maintaining a constant noise level of $\epsilon=5\%$. The remaining parameters, namely $W$, $N_m$, and $P$, remain fixed as detailed in Example \ref{exm1}. The results demonstrate the effectiveness of the proposed algorithm in handling a highly oscillating source function.

\begin{example}\label{exm3}
Consider a discontinuous function
\begin{equation*}
f(t)= \begin{cases}
0, & t \in[0, 4\pi/5), \\
2, & t \in[4\pi/5, 8\pi/5) ,\\
0.5, & t \in[8\pi/5, 12\pi/5) ,\\
1.5, & t \in[12\pi/5, 16\pi/5) ,\\
0, & t \in[16\pi/5, 4\pi).
\end{cases}
\end{equation*}
\end{example}

\begin{figure}
  \centering
  \includegraphics[width=0.3\textwidth]{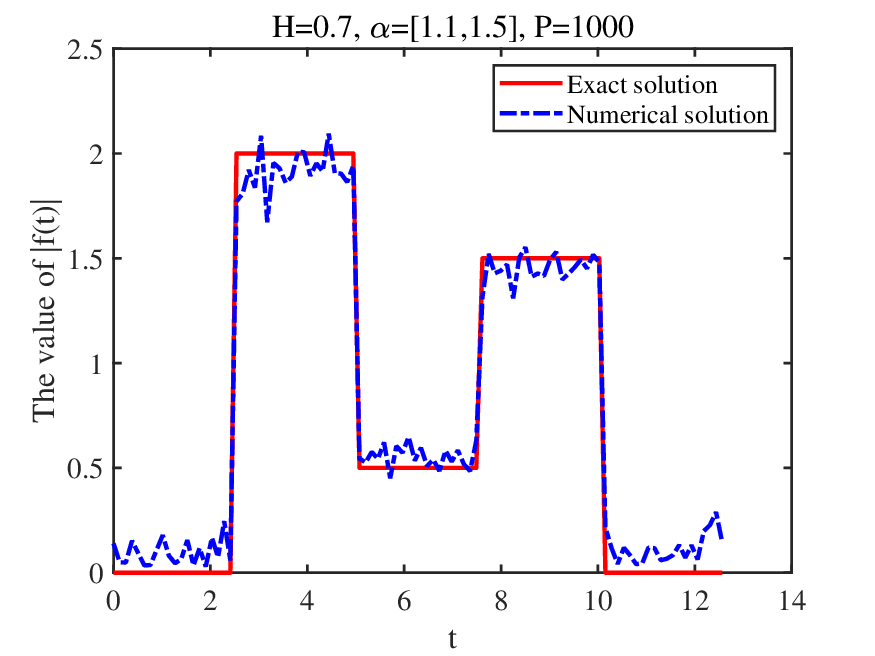}
  \includegraphics[width=0.3\textwidth]{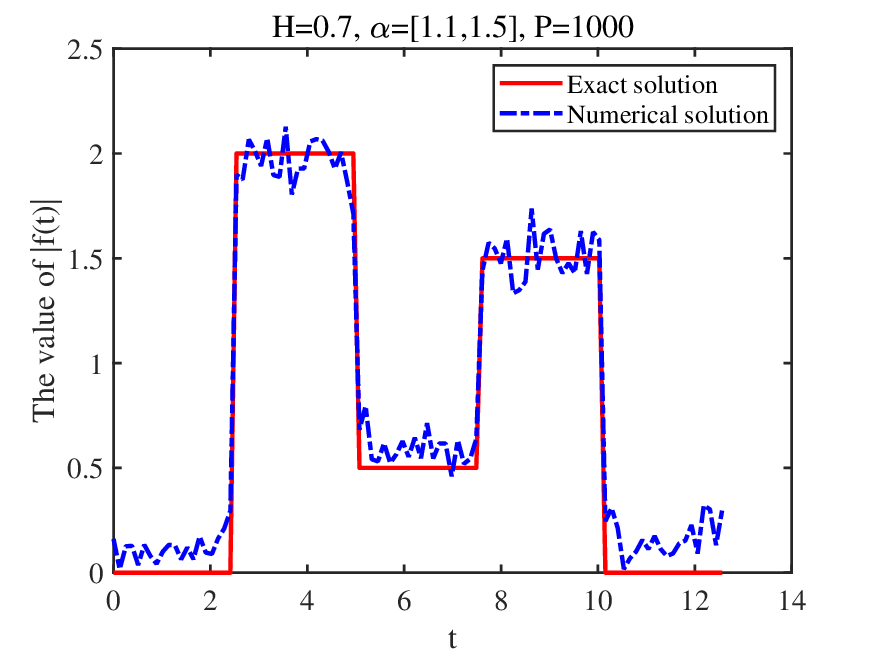}
  \includegraphics[width=0.3\textwidth]{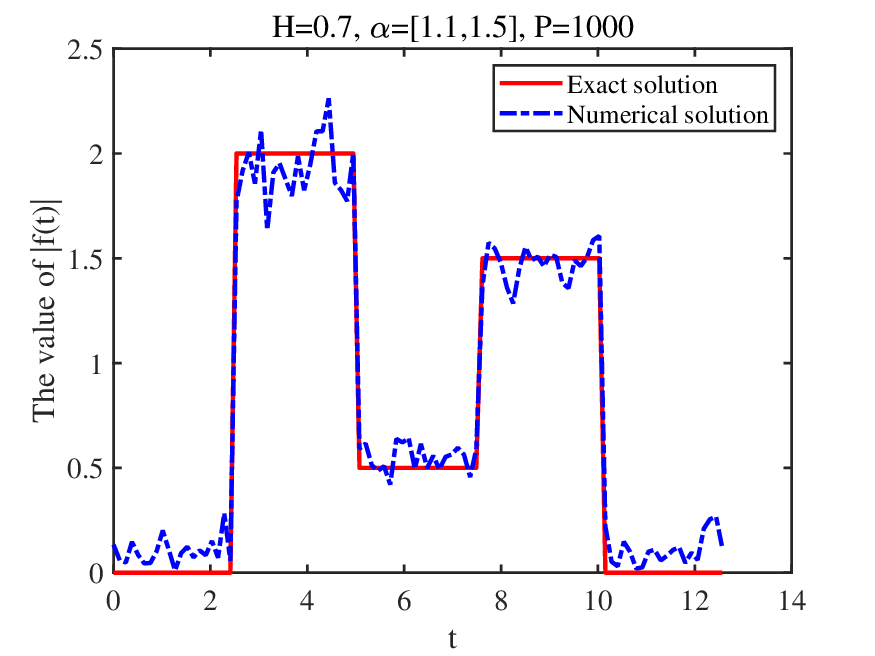}
  \caption{Example \ref{exm3}: Reconstruction of $|f(t)|$ with different combinations: (left) $\epsilon=1\%$, (middle) $\epsilon=5\%$, and (right) $\epsilon=10\%$, under the constant Hurst parameter $H=0.7$ and $\bm{\alpha}=[1.1,1.5]$.}\label{e3}
\end{figure}

This example was also examined in the study conducted by \cite{gong2021numerical}, and it is challenging to reconstruct due to the presence of infinitely many Fourier modes, with the corresponding Fourier coefficients decaying slowly. Once again, we will not provide an exhaustive investigation into the influence of various parameters on the reconstruction. In Figure \ref{e3}, numerical results showing the reconstruction of $|f(t)|$ in Example \ref{exm3} are presented. The parameters used for this representation include $H=0.7$, $\bm{\alpha}=[1.1,1.5]$, $W=3\pi$, $N_m=60$, $P=1000$, and varying noise levels $\epsilon=1\%, 5\%, 10\%$. Despite the emergence of the Gibbs phenomenon, a common occurrence when recovering discontinuous functions through Fourier transform based methods, the proposed algorithm demonstrates strong performance in handling the discontinuous case.

\section{Conclusion}\label{sec:6}

This paper addresses both the direct and inverse source problems associated with the stochastic multi-term time-fractional diffusion-wave equation. Regarding the direct random source problem, the well-posedness is obtained by demonstrating the well-posedness of its counterpart in the frequency domain. Furthermore, an analysis is conducted concerning the uniqueness and instability of the inverse random source problem in the frequency domain. To reconstruct the source function in the time domain, the PhaseLift method, combined with the spectral cut-off regularization technique, is utilized for numerical implementation. The numerical results validate the effectiveness of the proposed method.

This work expands upon existing results related to inverse random source problems for stochastic time-fractional differential equations, addressing more general cases. Specifically, it contains (1) both sub-diffusion cases with $\alpha_i\in(0,1)$ and super-diffusion cases with $\alpha_i\in(1,2)$, and (2) spatial random noise, which can be represented by fractional Brownian motion noise with $H\in(0,1)$ as opposed to the traditional Gaussian white noise with $H=\frac12$. Several challenges remain unresolved, including inverse random source problems in higher dimensions and inverse random potential problems for time-dependent stochastic partial differential equations, among others. We anticipate providing updates on our progress in addressing these challenges in future publications.

\end{document}